\documentclass[11pt,reqno,a4paper]{amsart}
\usepackage{aligned-overset}
\usepackage{mathrsfs, graphics, xfrac}
\usepackage{amssymb}
\usepackage[notref,notcite,final]{showkeys}
\usepackage{hyperref}\hypersetup{colorlinks=true, citecolor=blue, linkcolor=black}
\usepackage{graphicx, epstopdf}\graphicspath{{./immaginiprova/}}
\usepackage{bbm}
\usepackage{caption}
\usepackage{stmaryrd}
\usepackage{tikz}
\usetikzlibrary{decorations.markings}
\usetikzlibrary{arrows}
\usepackage{longtable}
\usepackage{afterpage}
\usepackage{epstopdf}

\numberwithin{equation}{section}

\newtheorem{theorem}{Theorem}[section]
\newtheorem{definition}[theorem]{Definition}
\newtheorem{lemma}[theorem]{Lemma}

\newtheorem{proposition}[theorem]{Proposition}

\newtheoremstyle{mytheorem}
{}
{}
{\it}
{}
{\bf}
{.}
{ }
{\thmnumber{#2.~}\thmname{#1}\thmnote{~\rm#3}}

\newtheoremstyle{myremark}
{}
{}
{}
{}
{\it}
{.}
{ }
{\thmnumber{#2.~}\thmname{#1}\thmnote{~\rm#3}}

\newtheoremstyle{myparagraph}
{}
{}
{\rm}
{\parindent}
{\bf}
{.}
{ }
{\thmnumber{#2.~}\thmname{#1}\thmnote{#3}}

\theoremstyle{myremark}
\newtheorem{remark}[subsection]{Remark}

\theoremstyle{myparagraph}

\newtheorem*{parag*}{}

\makeatletter
\def\@secnumfont{\sc}
\def\section{\@startsection{section}{1}%
\z@{1.5\linespacing\@plus .2\linespacing}{.7\linespacing}%
{\normalfont\sc\centering}}
\makeatother

\makeatletter
\def\ps@headings{\ps@empty
 \def\@evenhead{%
  \setTrue{runhead}%
  \normalfont\footnotesize
  \rlap{\thepage}\hfil
  \def\thanks{\protect\thanks@warning}%
  \leftmark{}{}\hfil}%
 \def\@oddhead{%
  \setTrue{runhead}%
  \normalfont\footnotesize\hfil
  \def\thanks{\protect\thanks@warning}%
  \rightmark{}{}\hfil \llap{\thepage}}%
\let\@mkboth\markboth}
\makeatother
\makeatletter
\renewenvironment{proof}[1][\proofname]{\par
  \pushQED{\qed}%
  \normalfont \topsep6\p@\@plus6\p@\relax
  \trivlist
  \itemindent\normalparindent
  \item[\hskip\labelsep
    \bfseries
    #1\@addpunct{.}]\ignorespaces
}{%
  \popQED\endtrivlist\@endpefalse
}
\providecommand{\proofname}{Proof}
\makeatother
\newcommand{\Flat}{\mathbb{F}}
\newcommand{\Mass}{\mathbb{M}}
\newcommand{\TP}{\textbf{TP}}
\newcommand{\OTP}{\textbf{OTP}}

\newcommand{\BR}{\textbf{BR}}
\newcommand{\Nor}{\mathscr{N}}
\newcommand{\Rect}{\mathscr{R}}\newcommand{\MM}{\mathbb{M}^\alpha}

\newcommand{\R}{\mathbb{R}}
\newcommand{\I}{\mathscr{I}}

\newcommand{\N}{\mathbb{N}}
\newcommand{\Z}{\mathbb{Z}}

\newcommand{\Po}{\mathscr{P}}
\newcommand{\D}{\mathscr{D}}

\newcommand{\Haus}{\mathscr{H}}

\newcommand{\M}{\mathscr{M}}

\newcommand{\eps}{\varepsilon}

\newcommand{\Lip}{\mathrm{Lip}}

\newcommand{\dist}{\mathrm{dist}}
\newcommand{\supp}{\mathrm{supp}}

\newcommand{\dV}{d_V\kern-1pt}

\newcommand{\trait}[3]{\vrule width #1ex height #2ex depth #3ex}

\newcommand{\trace}{\mathchoice%
  {\mathbin{\trait{.12}{1.2}{.03}\trait{.8}{0.09}{0.03}}}
  {\mathbin{\trait{.12}{1.2}{.03}\trait{.8}{0.09}{0.03}}}
  {\mathbin{\hskip.15ex\trait{.09}{.84}{0.02}\trait{.56}{.07}{.02}}\hskip.15ex}
  {\mathbin{\trait{.07}{.6}{.01}\trait{.4}{.06}{.01}}}}

\makeatletter
\@addtoreset{footnote}{section}
\makeatother

\newenvironment{itemizeb}
{\begin{itemize}\itemsep=2pt}{\end{itemize}}

\begin{document}

	%
\pagestyle{empty}
\pagestyle{myheadings}
\markboth%
{\underline{\centerline{\hfill\footnotesize%
\textsc{Gianmarco Caldini, Andrea Marchese,  \and Simone Steinbr\"uchel}%
\vphantom{,}\hfill}}}%
{\underline{\centerline{\hfill\footnotesize%
\textsc{Generic uniqueness of optimal transportation networks}%
\vphantom{,}\hfill}}}

	%
\thispagestyle{empty}

~\vskip -1.1 cm

	%

\vspace{1.7 cm}

	%
{\large\bf\centering
Generic uniqueness of optimal transportation networks\\
}

\vspace{.6 cm}

	%
\centerline{\sc Gianmarco Caldini, Andrea Marchese,  \and Simone Steinbr\"uchel}

\vspace{.8 cm}

{\rightskip 1 cm
\leftskip 1 cm
\parindent 0 pt
\footnotesize

	%
{\sc Abstract.}
We prove that for the generic boundary, in the sense of Baire categories, there exists a unique minimizer of the associated optimal branched transportation problem.


\par
}


{
  \hypersetup{linkcolor=black}
  \tableofcontents
}
\section{Introduction}
Let $K\subset\R^d$ be a convex compact set. Given two Radon measures $\mu_-$ (source) and $\mu_+$ (target) on $K$ with the same mass, i.e. $\Mass(\mu_-)=\Mass(\mu_+)$, a \emph{transport path} transporting $\mu_-$ onto $\mu_+$ is a rectifiable 1-current $T$ on $K$ whose boundary is the 0-current $\partial T=\mu_+-\mu_-$. For the precise definition of these objects we refer to Section \ref{s:preliminaries}. The current $T$ can be identified with a vector valued measure on $K$ which we denoted by $T$ as well. It can be written as $T=\vec T\theta\Haus^1\trace E$, where $E$ is a 1-rectifiable set, $\theta\in L^1(\Haus^1\trace E)$ and $\vec T$ is a unit vector field spanning the tangent ${\rm{Tan}}(E,x)$ at $\Haus^1$-a.e. $x\in E$. The constraint $\partial T=\mu_+-\mu_-$ is equivalent to the condition that the vector valued measure $T$ has distributional divergence which is a signed Radon measure and satisfies div$(T)=\mu_--\mu_+$. Given $\alpha\in(0,1)$, the \emph{ $\alpha$-mass} of a transport path $T$ as above is defined as 
$$\MM(T)=\int_E|\theta|^\alpha d\Haus^1.$$
We denote by $\D_k(K)$ the set of $k$-dimensional currents with support in $K$, and letting $b:=\mu_+-\mu_-\in \D_0(K)$, we denote by $\OTP(b)$ the set of minimizers of the \emph{optimal branched transportation problem} with boundary $b$, namely the minimizers of the $\alpha$-mass among rectifiable 1-currents $T$ with boundary $\partial T=b$.\\ 

For $\alpha\leq 1-1/d$, there are boundaries $b$ such that $\OTP(b)$ degenerates to the set of all currents $T$ with boundary $\partial T=b$, since there is no 1-current $T$ with $\partial T=b$ and $\MM(T)<\infty$, see \cite{morsant}. In turn, it is well known that there are boundaries $b$ such that $\OTP(b)$ contains more than one element of finite $\alpha$-mass; for instance one can exhibit a non-symmetric minimizer $T$ for which $\partial T$ is symmetric, so that the network $T'$ symmetric to $T$ is a different minimizer (see Figure \ref{f:1}).

\begin{figure}[ht]
	\centering 
    \includegraphics[width=0.45\textwidth]{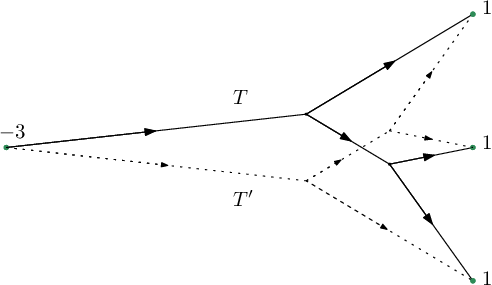}
    \caption{The boundary $\partial T$ is symmetric (with respect to the horizontal axis) and $T\in\OTP(\partial T)$ is not symmetric, hence the symmetric copy $T'$ is a different minimizer.}\label{f:1}
\end{figure}
The aim of this paper is to prove that for the generic boundary, in the sense of Baire categories, the associated optimal branched transportation problem has a unique minimizer.
To this purpose, we denote the set of boundaries by
\[ {\mathscr{B}}_{0}(K) := \{b\in{\mathscr{D}}_{0}(K): \text{there is an } S \in {\mathscr{D}}_{1}(K) \text{ with } \partial S= b\}\,, \]
we fix an arbitrary constant $C>0$ and we define 
\begin{equation}\label{def_AC}
    A_C:=\{b\in{\mathscr{B}}_{0}(K): \mathbb{M}(b)\leq C\;{\rm{and}}\;\MM(T)\leq C\mbox{ for every }T\in\OTP(b)\}.
\end{equation}
We metrize $A_C$ with the \emph{flat norm} $\Flat_K$, see \eqref{e:flat def}, and we observe that the set $A_C$ endowed with the induced distance is a non-trivial complete metric space, see Lemma \ref{l:closed}.
Our main result is the following

\begin{theorem}\label{t:main}
    The set of boundaries $b\in A_C$, for which $\OTP(b)$ is a singleton, is residual.
\end{theorem}
\subsection{Previous results on the well-posedness of the problem}
The variational formulations of the optimal branched transportation problem were inspired by the discrete model introduced by Gilbert in \cite{Gilbert} and are used to model supply and demand transportation systems which naturally show ramifications as a result of a transportation cost which favors large flows and penalizes diffusion. 

In our paper, we adopt the  \emph{Eulerian formulation} proposed by Xia in \cite{Xia}. Due to a celebrated result by Smirnov, see \cite{Smirnov93}, this is equivalent to the \emph{Lagrangian formulation}, introduced by Maddalena, Morel and Solimini in \cite{MSM}, see \cite{BCM,Pegon}. Existence results and some regularity properties of minimizers have been established for instance in
\cite{BeCaMo, BraSol, morsant, xia2, xiaBoundary}. Recently, another helpful well-posedness property of the problem was established in \cite{CDRMcpam}: the stability of minimizers with respect to variations of the boundary, see \cite{CDRMPP} for the Lagrangian counterpart. Slightly improving upon the main result of \cite{CDRMcpam}, see Theorem \ref{t:stability_new}, we advance on the study of the well-posedness properties of the branched transportation problem, as we establish the first result on the generic uniqueness of minimizers, in full generality, namely in every dimension $d$ and for every exponent $\alpha\in(0,1)$.

Prior to our work, we are aware of only one elementary result on the uniqueness of minimizing networks. It appeared in the original paper by Gilbert \cite{Gilbert}, and says that there exists at most one discrete \emph{minimum cost communication network} with a given Steiner topology.  

Several variants and generalizations of the branched transportation problem were proposed and studied by many authors in recent years, see for instance \cite{BF, BrBuSa, BW, BW1, brabutsan, BrGaReSun, BrPaSun, BrSun, CDRMjmpa, MMST, MMT, PaoliniStepanov, XiaXu}. For the sake of simplicity, we prove the generic uniqueness of minimizers only for the Eulerian formulation introduced in \cite{Xia}.

\subsection{Strategy of the proof}
Using a small modification of the stability property proved in \cite{CDRMcpam}, see Theorem \ref{t:stability_new}, we show that in order to prove Theorem \ref{t:main}, it suffices to prove the \emph{density} of the set of boundaries $b\in A_C$ for which $\OTP(b)$ is a singleton, see Lemma \ref{l:residual}. A similar reduction principle is used in \cite{MorganInv, Morgan} to prove that the generic (higher dimensional) boundary spans a unique minimal hypersurface. 

The proof of the density result is based on the following perturbation argument. Firstly, we prove that we can reduce to a finite atomic boundary $b$ with integer multiplicities, exploiting the fact that multiples of such boundaries are dense in $A_C$, see Lemma \ref{l:integral_bdry}. For these boundaries, we prove that the solutions to the optimal branched transportation problem are multiples of polyhedral integral currents, see Lemma \ref{l:integral_minimizers}. Then we improve the uniqueness result of \cite{Gilbert} to suit the discrete branched transportation problem, obtaining as a byproduct that for every finite atomic boundary $b$ as above the set $\OTP(b)$ is finite, see Lemma \ref{l:finiteminimizers}. We deduce the existence of a set of points $\{p_1,\dots,p_h\}$ in the regular part of the support of a fixed transport path $T\in\OTP(b)$ with the property that $T$ is the only element in $\OTP(b)$ whose support contains $\{p_1,\dots,p_h\}$, see Lemma \ref{l:magic_points}. 

Next, we aim to ``perturb" the boundary $b$ close to the points $p_1,\dots,p_h$ in order to obtain boundaries with unique minimizers, keeping in mind the fact that the perturbed boundaries should not escape from the set $A_C$. More in detail, we define a sequence $(b_n)_{n\geq 1}\subset A_C$ of boundaries for the optimal branched transportation problem with the property that $\Flat_K(b_n-b)\to 0$ as $n\to\infty$. Moreover, each $b_n$ has points of its support (with small multiplicity) in proximity of $p_1,\dots,p_h$, so that every minimizing transport path $S_n$ with boundary $\partial S_n=b_n$ is forced to have such close-by points in its support. Exploiting again the stability property of Theorem \ref{t:stability_new}, we deduce that for every choice of $S_n\in\OTP(b_n)$ there exists $S\in\OTP(b)$ such that, up to subsequences, it holds $\Flat_K(S_n-S)\to 0$ and we can infer the Hausdorff convergence of the supports of the $S_n$'s to the union of the support of $S$ and the points $p_1,\dots,p_h$, see Lemma \ref{l:conv_hauss}. Notice that at this stage we cannot deduce from Lemma \ref{l:magic_points} that $S=T$, since the portion of $S_n$ which is in proximity of some of the $p_i$'s might vanish in the limit. In order to exclude this possibility, we perform a fine analysis of the structure of the network $S_n$ around the points $p_1,\dots, p_h$, see \S \ref{Step2}: this allows us to exclude all possible local topologies except for two, see \eqref{e:THE_ONE}, proving that $p_1,\dots,p_h$ are contained in the support of $S$ (so that in particular $S=T$ by Lemma \ref{l:magic_points}) and that $\OTP(b_n)=\{S_n\}$, for $n$ sufficiently large, see Lemma \ref{l:main}, which concludes the proof of Theorem \ref{t:main}.\\ 

\begin{remark}
It is much easier to prove just density in $\bigcup_{C>0}A_C$ of the boundaries $b$ for which $\OTP(b)$ is a singleton. Indeed, it is significantly simpler to perform the strategy outlined above if one is allowed to choose $b_n$ simply satisfying $\Flat_K(b_n-b)\to 0$ and $\MM(S_n)\leq C$, but possibly with $\Mass(b_n)>C$: for instance it suffices to choose the perturbation $b_n$ as in \eqref{e:bienne} with $k=1$, in which case it is easy to prove that $\OTP(b_n)$ is a singleton. Obviously such type of perturbation is not admissible in order to prove the residuality result of Theorem \ref{t:main}, since such boundaries $b_n$ do not belong to $A_C$. One of the challenges in our proof is therefore to find suitable perturbations $b_n$ of $b$ which are internal to the set $A_C$ and such that for the boundary $b_n$ there exists a unique minimizer of the optimal branched transportation problem, for $n$ sufficiently large.
\end{remark}

\begin{remark}
Following \cite{MorganInv, Morgan}, it would be tempting to adopt a seemingly simpler strategy to prove Theorem \ref{t:main}. Indeed the density result would be an easy consequence of the following \emph{unique continuation principle}: if $b$ is a finite atomic boundary with integer multiplicities, then any two elements of $\OTP(b)$ which coincide on a neighbourhood of the support of $b$ necessarily coincide globally.\\
One reason to believe that such a statement could be true is the fact that, knowing the directions and the multiplicities of all the edges colliding at a branch point except for one, it is possible to deduce the information on the last edge, by exploiting a balancing condition which is due to the stationarity of the network for the $\alpha$-mass. The main obstruction to prove the statement is the following. If for a minimizer $T$ in $\OTP(b)$ two or more edges emanating from the boundary collide at some branch point, it is not obvious that for another minimizer $T'$ which coincides with $T$ on a neighbourhood of the support of $b$ the same edges still collide: it might happen that $T'$ has some branch point in the interior of one of these edges. We do not exclude that the statement could be true, but we believe that this cannot be proved only by \emph{local} properties, which would make a potential proof quite involved. This is the reason why we opted for a completely different strategy, which is based ultimately on local arguments only.\\
The presence of singularities is not an issue in the framework of minimal surfaces, because the singular set is too small to disconnect the regular part of the surface. We believe that the strategy which we devised is of general interest and can be adapted to prove generic uniqueness of solutions to other variational problems with \emph{large} singular sets, see \cite{DLHMS, DLHMS2}.\\
\end{remark}

\section{Preliminaries}\label{s:preliminaries}
Through the paper $K\subset$ $\R^d$ denotes a convex compact set. We denote by $\M(K)$ the space of signed Radon measures on $K$ and by $\M_+(K)$ the subspace of positive measures. The \emph{total variation measure} associated to a measure $\mu\in\M(K)$ is denoted by $\|\mu\|$ and $\mu_+:=1/2(\|\mu\|+\mu)$ and $\mu_-:=1/2(\|\mu\|-\mu)$ denote respectively the positive and the negative part of $\mu$. The \emph{mass} of $\mu$ is the quantity $\Mass(\mu):=\|\mu\|(K)$. We say that a measure is \emph{finite atomic} if its support is a finite set.

We adopt Xia's \emph{Eulerian formulation} \cite{Xia} of the optimal branched transportation problem. This employs the theory of currents, for which we refer the reader to \cite{FedererBOOK}. We recall that a $k$-dimensional \emph{current} on $\R^d$ is a continuous linear functional on the space $\D^k(\R^d)$ of smooth and compactly supported differential $k$-forms and we denote by $\D_k(K)$ the space of $k$-dimensional currents with support in $K$. 
The space $\D_k(K)$ is endowed with a norm which is called \emph{mass} and denoted by $\Mass$. By the Riesz representation theorem, a current $T$ with $\Mass(T)<\infty$ can be identified  with vector-valued Radon measures $\vec T\mu_T$ where $\vec T$ is a unit $k$-vector field and $\mu_T$ a positive Radon measure. The mass of the current $T$ coincides with the mass of the measure $\mu_T$. We denote by $\supp(T)$ the \emph{support} of a current $T$, which coincides with the support of the measure $\mu_T$, if $T$ has finite mass. 
The \emph{boundary} of a current $T\in\D_k(K)$ is the current $\partial T\in\D_{k-1}(K)$ such that $$\partial T(\phi)=T(d\phi), \quad\mbox{ for every $\phi\in\D^{k-1}(\R^d)$}.$$
A current $T$ such that $\Mass(T)+\Mass(\partial T)<\infty$ is called a \emph{normal} current. The space of $k$-dimensional normal currents with support in $K$ is denoted by $\Nor_k(K)$.

We say that a current $T\in\D_k(K)$ is \emph{rectifiable} and we write $T\in\Rect_k(K)$ if we can identify $T$ with a triple $(E,\tau,\theta)$, where $E\subset K$ is a $k$-rectifiable set, $\tau(x)$ is a unit $k$-vector spanning the tangent space Tan$(E,x)$ at $\Haus^k$-a.e. $x$ and $\theta\in L^1(\Haus^k\trace E)$, where the identification means that
$$T(\omega)=\int_E\langle\omega(x),\tau(x)\rangle \theta(x) d\Haus^k(x), \quad \mbox{ for every $\omega\in \D^k(\R^d)$}.$$

Those currents $T=(E,\tau,\theta)$ which are normal and rectifiable with integer multiplicity $\theta$ are called \emph{integral} currents. The subgroup of integral currents with support in $K$ is denoted by $\I_k(K)$.

A $k$-dimensional \emph{polyhedral} current is a current $P$ of the form  
\begin{equation}\label{e:poly}
P:=\sum_{i=1}^N\theta_i\llbracket \sigma_i\rrbracket,
\end{equation}
where $\theta_i\in\R\setminus\{0\}$, $\sigma_i$ are nontrivial $k$-dimensional simplexes in $\R^d$, with disjoint relative interiors, oriented by $k$-vectors $\tau_i$ and $\llbracket \sigma_i \rrbracket=(\sigma_i,\tau_i,1)$ is the multiplicity-one rectifiable current naturally associated to $\sigma_i$. 
The subgroup of polyhedral currents with support in $K$ is denoted $\Po_k(K)$. A polyhedral current with integer coefficients $\theta_i$ is called \emph{integer polyhedral}.\\

Given $\alpha\in(0,1)$ and a 1-current $T\in\Nor_1(K)\cup\Rect_1(K)$ we define the $\alpha$-mass
\begin{equation*}
    \MM(T):=
    \begin{cases}
      \int_E|\theta|^\alpha d\Haus^1, & \text{if}\ T=(E,\tau,\theta)\in\Rect_1(K); \\
      +\infty, & \text{otherwise.}
    \end{cases}
  \end{equation*}
If $\mu_-$ and $\mu_+$ are elements of $\M_+(K)$ such that $\Mass(\mu_-)=\Mass(\mu_+)$, the \emph{optimal branched transportation problem} with boundary $b=\mu_+-\mu_-$ seeks a normal current $T\in\Nor_1(K)$ which minimizes the $\alpha$-mass $\Mass^\alpha$ among all currents $S$ with boundary $\partial S=b$. Hence, we denote by $\TP(b)$ the set of \textit{transport paths} with boundary $b$ as $$\TP(b):=\{T\in\Nor_1(K) : \partial T=b\},$$ and the \textit{least transport energy} associated to $b$ as $$\mathbb{E}^{\alpha}(b):= \inf\{\MM(T): T \in \TP(b)\}.$$ We define the set of \textit{optimal transport paths} with boundary $b$ by $$\OTP(b):= \{T\in \TP(b) :\MM(T)=\mathbb{E}^{\alpha}(b)\}.$$ 

Let $A_C$ be the set of boundaries defined in \eqref{def_AC}. Due to the Baire category theorem, the next lemma ensures that a \emph{residual} subset of $A_C$ (namely a set which contains a countable intersection of open dense subsets) is dense. We recall that the flat norm $\Flat_K(T)$ of a current $T\in\D_k(K)$ is the following quantity, see \cite[\S 4.1.12]{FedererBOOK},
\begin{equation}\label{e:flat def}
   \Flat_K(T):=\inf\{\Mass(T-\partial S)+\Mass(S):S\in\D_{k+1}(K)\}.
\end{equation}

\begin{lemma}\label{l:closed}
The set $A_C$ is $\mathbb{F}_K$-closed. In particular $(A_C,\mathbb{F}_K)$ is a complete metric space.
\end{lemma}
\begin{proof}
The second part of the statement follows from the first part and from the $\mathbb{F}_K$-compactness of 0-currents with support in $K$ and mass bounded by $C$, see \cite[\S 4.2.17]{FedererBOOK}.

In order to prove that $A_C$ is $\Flat_K$ closed, let $(b_j)_{j\in\N}$ be a sequence of elements of $A_C$ and let $b$ be such that $\mathbb{F}_K(b_j-b)\to 0$ as $j\to\infty$. We want to prove that $b\in A_C$. By the lower semicontinuity of the mass (with respect to the flat convergence), we have $\Mass(b)\leq C$. For any $j\in\N$, let $T_j \in \OTP(b_j)$. By \cite[Proposition 3.6]{CDRMcalcvar}, we have $\mathbb{M}(T_j)\leq C^{1-\alpha}\mathbb{M}^\alpha(T_j)\leq C^{2-\alpha}$.
By the compactness theorem for normal currents, there exists $T\in{\mathscr{N}}_{1}(K)$ such that, up to (non relabeled) subsequences $\mathbb{F}_K(T_j-T)\to 0$. By the continuity of the boundary operator we have $\partial T=b$ and by the lower semicontinuity of the $\alpha$-mass, see \cite{CDRMS}, we have $\mathbb{M}^\alpha(T)\leq C$ and hence $b\in A_C$. 
\end{proof}

Consider the following subset of $A_C$, which represents the set of boundaries admitting non-unique minimizers:
$$NU_C:=\{b\in A_C: \exists\,\, T^1, T^2 \in\OTP(b) \text{ such that } T^1\neq T^2\}.$$
Notice that since $b\in A_C$ then $\MM(T^1)=\MM(T^2)\leq C$. We have the following:

\begin{lemma}\label{l:residual}
Assume that the set $A_C\setminus NU_C$ is $\mathbb{F}_K$-dense in $A_C$. Then it is residual.
\end{lemma}
\begin{proof}
For $m \in \N \setminus \{0\}$, consider the sets 
$$NU_C^m:=\{b\in A_C: \exists\, \{T^1,T^2\}\subset\OTP(b)\;{\rm with}\;\mathbb{F}_K(T^2-T^1)\geq m^{-1}\}.$$
Since $NU_C^m\subset NU_C$, then $(A_C\setminus NU_C^m)\supset(A_C\setminus NU_C)$ and hence, by assumption, $A_C\setminus NU_C^m$ is $\mathbb{F}_K$-dense in $A_C$ for every $m$. Therefore $NU_C^m$ has empty interior in $A_C$ for every $m$.

To conclude, it is sufficient to prove that $NU_C^m$ is closed for every $m$. Consider a sequence $(b_j)_{j\in\N}$ of elements of $NU_C^m$ and let $b\in A_C$ be such that $\mathbb{F}_K(b_j-b)\to 0$. We need to prove that $b\in NU_C^m$. For every $j\in\N$, take 
$$\{T^1_j,T^2_j\}\subset\OTP(b_j)\quad {\mbox{ with }}\quad \mathbb{F}_K(T^2_j-T^1_j)\geq m^{-1}.$$ 
As in the proof of Lemma \ref{l:closed}, we deduce that there exist $T^1,T^2\in{\Nor}_{1}(K),$ such that $\partial T^1=\partial T^2=b$ and, up to (non relabeled) subsequences, $\mathbb{F}_K(T^1_j-T^1)\to 0$, $\mathbb{F}_K(T^2_j-T^2)\to 0$ as $j\to\infty$. Clearly $\mathbb{F}_K(T^2-T^1)\geq m^{-1}$. By Theorem \ref{t:stability_new}, we have $\{T_1, T_2\}\subset\OTP(b)$, hence $b\in NU_C^m$.
\end{proof}

\section{Density of boundaries with unique minimizers: preliminary reductions}
\subsection{Reduction to integral boundaries and integer polyhedral minimizers}
\begin{lemma}\label{l:integral_bdry}
    For any $b \in A_C$ and $\eps >0$, there exist $\delta>0$ and a boundary $b'' \in A_{C-\delta}$ with
    \[ \Flat_K(b-b'') < \eps  \qquad \text{ and } \qquad
    b'' = \eta b_I\]
    for some $\eta >0$ and $b_I\in\I_0(K)$.
\end{lemma}
\begin{proof}
Without loss of generality and up to rescaling, we can assume $C=1$ and write $A$ instead of $A_C$. Let $b \in A$ and $T\in\OTP(b)$ and define $T_{\varepsilon}:= (1-\varepsilon/4) T$. Then $b_{\varepsilon}:=\partial T_{\varepsilon}=(1-\varepsilon/4)b$ and $T_{\varepsilon}\in\OTP(b_{\varepsilon})$ satisfy 
\begin{equation}\label{e:estTepsbeps}
\mathbb{M}^{\alpha} (T_{\varepsilon})\le (1 - \varepsilon/4)^{\alpha}\quad\mbox{ and }\quad \Mass(b_\varepsilon)\leq 1 - \varepsilon/4.  
\end{equation}
Since we also have $\mathbb{M}(b-b_{\varepsilon}) \le \varepsilon/4$, we deduce that 
\begin{equation}\label{e:estflatbeps}
\mathbb{F}_K(b-b_{\varepsilon}) \le \varepsilon/4.    
\end{equation}
Now apply \cite[Theorem 5]{MW} to obtain, possibly after rescaling, a current $T'_\eps \in\Po_1(K)$, such that, denoting $b'_\eps:=\partial T'_\eps$, we have 
\begin{equation}\label{e:estTprimo}
\mathbb{M}^{\alpha}(T'_\eps)\leq \mathbb{M}^{\alpha}(T_{\varepsilon}),\quad\mathbb{M}(b'_\eps) \leq \mathbb{M}(b_{\varepsilon})\quad\mbox{and}\quad \Flat_K(T'_\eps-T_\eps)\leq \varepsilon/4,    
\end{equation}
and in particular $\Flat_K(b'_\eps-b_\varepsilon)\leq \varepsilon/4$. We can write 
$$T'_\eps = \sum_{i=1}^{N} \theta_i' \llbracket\sigma_i\rrbracket$$ as in \eqref{e:poly}. Up to changing the orientation of $\llbracket\sigma_i\rrbracket$, we may assume $\theta'_i>0$ for every $i$. Fix $\eta := \varepsilon/(16N)$ and denote 
$\theta_i'' := \eta  \left\lfloor \frac{\theta_i'}{\eta} \right\rfloor$ (where $\lfloor x\rfloor$ is the largest integer smaller than or equal to $x$) so that
\begin{equation}\label{e:appx}
     0\leq\theta_i'-\theta_i''< \frac{\varepsilon}{16N} \qquad\mbox{for every $i \in \{1,\dots,N\}$}.
      \end{equation} Define $$T'':= \sum_{i=1}^{N}\theta_i'' \llbracket\sigma_i\rrbracket$$ and denote $b''=\partial T''$. Observe that by \eqref{e:appx} and \eqref{e:estTprimo} we have 
\begin{equation}\label{e:alphamasstsecondo}
    \MM(T'')\leq \MM(T'_\eps)\leq \MM(T_\varepsilon)< 1.
\end{equation} 
      
For every $i\in \{1,\dots,N\}$, we denote by $x_i$ and $y_i$ respectively the first and second endpoint of the oriented segment $\sigma_i$, so that we can write 
$$b'_\eps= \sum_{i=1}^{N} \theta_i' (\delta_{y_i} - \delta_{x_i})$$ 
which we can rewrite as 
$$b'_\eps= \sum_{j=1}^{M}\beta_j' \delta_{z_j},$$
where all points $z_j$ are distinct and
$$\beta_j':= \left( \sum_{\{i: y_i = z_j \}} \theta_i' - \sum_{\{i: x_i = z_j\}} \theta_i'\right).$$
Analogously, we define $$\beta_j'':= \left( \sum_{\{i: y_i = z_j \}} \theta_i'' - \sum_{\{i: x_i = z_j\}} \theta_i''\right),$$
so that we can write
$$b''= \sum_{j=1}^{M}\beta_j'' \delta_{z_j}.$$
Thus we obtain
\begin{equation}\label{e:massb''b'}
\begin{split}
\Mass(b''-b'_\eps ) = \sum_{j=1}^{M}|\beta_j'' - \beta_j'|
\leq \Bigg|\sum_{\{i: x_i = z_j \}} (\theta_i''- \theta_i')\Bigg| + \Bigg|\sum_{\{i: y_i = z_j\}} (\theta_i''- \theta_i') \Bigg| \stackrel{(\ref{e:appx})}< \frac{\eps}{16} + \frac{\eps}{16} = \frac{\eps}{8}
\end{split}
\end{equation} 
and by \eqref{e:estTepsbeps} and \eqref{e:estTprimo} we deduce
\begin{equation}\label{e:massbsecondo}
    \Mass(b'') \leq \Mass(b'_\eps) +\Mass(b''-b'_\eps) <1. 
\end{equation}
Combining \eqref{e:massb''b'} with \eqref{e:estflatbeps} and \eqref{e:estTprimo}, we get
$\mathbb{F}_K(b-b'') < \varepsilon$. The conclusion follows denoting $b_I:=\eta^{-1}b''$ and observing that $b_I\in\I_0(K)$ (as the $\theta''_i$ are multiples of $\eta$) and that by \eqref{e:alphamasstsecondo} and \eqref{e:massbsecondo} we have $b''\in A_{1-\delta}$ for some $\delta>0$.
\end{proof}

\begin{lemma}
\label{l:integral_minimizers}
If $b\in\I_0(K)$ and $T\in\Nor_1(K)$ is in $\OTP(b)$ then $T\in \Po_1(K)\cap \I_1(K)$.
\end{lemma}
\begin{proof}
Combining the \emph{good decomposition} properties of optimal transport paths \cite[Proposition 3.6]{CDRMcalcvar} and their \emph{single path property} \cite[Proposition 7.4]{BCM} with the assumption $\partial T\in\Po_0(K)$, we deduce that there are finitely many Lipschitz simple paths $\gamma_1,\dots,\gamma_N$ of finite length such that $T$ can be written as a $T=\sum_{i=1}^N a_i\llbracket\gamma_i\rrbracket$, where $a_i>0$ for every $i$ and $\llbracket\gamma_i\rrbracket\in\I_1(K)$ is the current $({\rm{Im}}(\gamma_i), \gamma_i'/|\gamma_i'|, 1)$. Moreover, again by \cite[Proposition 7.4]{BCM}, one can assume that ${\rm{Im}}(\gamma_i)\cap {\rm{Im}}(\gamma_j)$ is connected for every $i,j$, which in turn implies that $T\in\Po_1(K)$. Hence we can write $$T:=\sum_{\ell=1}^N\theta_\ell\llbracket\sigma_\ell\rrbracket,$$
where $\sigma_\ell$ are non-overlapping oriented segments and $\theta_\ell\in\R$. We want to prove that $\theta_\ell\in\Z, \forall\ell$.

Denote $$\mathcal{I}:=\{\ell\in\{1,\dots,N\}:\theta_\ell\in\R\setminus\Z\}$$
and let $\hat T:=\sum_{\ell\in\mathcal{I}}\theta_\ell\llbracket\sigma_\ell\rrbracket$. Assume by contradiction that $\hat T\neq 0$.
Note that $T-\hat T\in\I_1(K)$ and therefore, since $b\in\I_0(K)$, we have $\partial\hat T=b-\partial(T-\hat T)\in\I_0(K)$. Hence, for every point $x$ in the support of $\partial\hat T$ there are at least two distinct segments $\sigma_{\ell_1}$ and $\sigma_{\ell_2}$ having $x$ as an endpoint. This implies that the support of $\hat T$, and in particular also the support of $T$, contains a loop, which contradicts \cite[Proposition 7.8]{BCM}.
\end{proof}
\subsection{Finiteness of the set of minimizers for an integral boundary}
\begin{definition}[Topology and branch points]\label{d:topo} \normalfont
Let $b\in\I_0(K)$ and let $T, T'\in\Po_1(K)$ with $\partial T=\partial T'=b$. We say that $T$ and $T'$ have the same \emph{topology} if there exist two ordered sets, each made of distinct points, $\{x_1,\dots,x_M\}$ and $\{x'_1,\dots,x'_M\}$ with the following properties:
\begin{itemizeb}
\item[(i)] for every $p\in\supp(b)$ there exists $i$ such that $x_i=p=x'_i$;
\item[(ii)] denoting $\sigma_{ij}$ the segment with first endpoint $x_i$ and second endpoint $x_j$ and $\sigma'_{ij}$ the segment with first endpoint $x'_i$ and second endpoint $x'_j$, $T$ and $T'$ can be written respectively as 
\begin{equation}\label{e:def_topo}
T=\sum_{i<j}a_{ij}\llbracket\sigma_{ij}\rrbracket,\quad T'=\sum_{i<j}a'_{ij}\llbracket\sigma'_{ij}\rrbracket, \quad \mbox{ for some $a_{ij}, a'_{ij}\in\R$.}    
\end{equation}
\item[(iii)] the representations in \eqref{e:def_topo}, restricted to the nonzero addenda, is of the same type as \eqref{e:poly}. In particular, if $a_{ij}$ and $a_{kl}$ (resp. $a'_{ij}$ and $a'_{kl}$) are nonzero, then $\sigma_{ij}$ and $\sigma_{kl}$ (resp. $\sigma'_{ij}$ and $\sigma'_{kl}$) have disjoint interiors. Moreover, the number of nonzero addenda in the representation of $T$ (resp. $T'$) given in \eqref{e:def_topo} coincides with the smallest number $N$ for which $T$ (resp. $T'$) can be written as in \eqref{e:poly}.
\item [(iv)]$a_{ij}=0$ if and only if $a'_{ij}=0$. In particular, the number $N$
of the previous point is the same for $T$ and $T'$.
\end{itemizeb}
One can check that the above conditions define an equivalence relation on the set of polyhedral currents. We call the \emph{topology} of a polyhedral current $T$ the corresponding equivalence class. Notice that the number $M$ depends only on the equivalence class and for every $T$ the (unordered) set $\{x_1,\dots,x_M\}$ is uniquely determined, by property (iii). The set $\{x_1,\dots,x_M\}\setminus \supp(b)$ is called the set of \emph{branch points} of $T$ and denoted by $\BR(T)$. By Lemma \ref{l:integral_minimizers}, for every $T\in\OTP(b)$ the topology of $T$ and the set $\BR(T)$ are well defined. 
\end{definition}

\begin{lemma}\label{l:numberBranch}
	Let $0\neq b\in\I_0(K)$ and $T \in \OTP(b)$. Then $ \Haus^0(\BR(T))\leq \Haus^0(\supp(b))-2$.
\end{lemma}
\begin{proof}
	Suppose without loss of generality that $\Haus^0(\BR(T))>0$. Assume by contradiction that the lemma is false and let $n$ be the minimal number such that there exist $b\in\I_0(K)$ and $T \in \OTP(b)$ such that $$\Haus^0(\BR(T))+2> n=\Haus^0(\supp(b)).$$ 
	Notice that $n>2$. 
	Fix $p \in \BR(T)$ and let $\varepsilon>0$ be such that $$(\overline B_{\varepsilon}(p)\setminus\{p\})\cap(\supp(b)\cup\BR(T))=\emptyset.$$ Denote by $T_1, \dots, T_m$ the restriction of $T$ to the connected components of $\supp(T)\setminus B_\varepsilon(p)$. We notice that $m \geq 3$. Indeed, if $m=1$ we would have the contradiction $p\in\supp(b)$ and if $m=2$, writing $T$ as in \eqref{e:def_topo}, the only two segments with nonzero coefficient having $p$ as an endpoint cannot be collinear by property (iii): this contradicts the the fact that $T\in\OTP(b)$. Observe that for every $i$ we have that $\supp(\partial T_i)\setminus \supp(b)$ consists of exactly one point $p_i$, so that
	\begin{equation}\label{b}
 n = \sum_{i=1}^m \left( \Haus^0(\supp(\partial T_i)) - 1 \right).
 \end{equation}
By minimality of $n$ and the fact that $m \geq 3$, we have \begin{equation}\label{i_s} \Haus^0(\BR(T_i)) \leq \Haus^0(\supp(\partial T_i)) -2 \quad \text{ for all $i \in \{1,\dots ,m\}$}. \end{equation} Since $m\geq 3$, the combination of \eqref{b} and \eqref{i_s} leads to a contradiction.
\end{proof}

\begin{lemma}\label{l:topo_and_support}
Let $b\in\I_0(K)$ and let $T,T'\in\Po_1(K)$ with $\partial T = b =\partial T'$ have the same topology. Assume moreover that $\supp(T)$ and $\supp(T')$ do not contain loops. Write $T$ and $T'$ as in \eqref{e:def_topo} with properties (i)-(iv) and with the same orientation on each segment. Then $a_{ij}=a'_{ij}$ for every $i,j$.
\end{lemma}

\begin{proof}
By contradiction, let $T,T'$ be nonzero currents with the same topology, $\partial T = b =\partial T'$, and minimizing the quantity $M$ in Definition \ref{d:topo} among all pairs for which the lemma is false. We claim that there exists a point $p\in\supp(b)$ and (up to reordering) indexes $i,j\in\{1,\dots,M\}$ such that
\begin{itemizeb}
\item [(a)] $a_{lj}=0=a'_{lj}$ for every $l\neq i$;
\item [(b)] $x_j=p=x'_j$ and $a_{ij}\neq a'_{ij}$, with $a_{ij},a'_{ij}\in\R\setminus\{0\}$. 
\end{itemizeb}
The validity of (a) follows from the absence of loops. On the other hand, if a point $p$ as in (a) violated (b), one could restrict the currents $T$ and $T'$ respectively to the complementary of $\sigma_{ij}$ and $\sigma'_{ij}$, thus contradicting the minimality of $M$.
The validity of (a) and (b) is a contradiction because the multiplicities $a_{ij}$ and $a'_{ij}$ correspond to the multiplicity of $p$ as point in the support of $b$. 
\end{proof}

\begin{lemma}\label{l:supports_determine_current}
Let $b\in\I_0(K)$ and $S,T\in\OTP(b)$ with $\supp(S)=\supp(T)$. Then $S=T$. 
\end{lemma}
\begin{proof}
Assume by contradiction $S\neq T$. By Lemma \ref{l:integral_minimizers}, $S-T\in\Po_1(K)\cap\I_1(K)$ is a nontrivial current with $\partial(S-T)=0$ and by assumption $\supp(S-T)\subset \supp(S)$. As in the proof of Lemma \ref{l:integral_minimizers} we deduce that $\supp(S-T)$ contains a loop. In particular, so does $\supp(S)$, which contradicts \cite[Proposition 7.8]{BCM}.
\end{proof}

\begin{lemma}\label{l:finiteminimizers}
    Let $b\in\I_0(K)$ be a boundary. Then $\OTP(b)$ is finite.
\end{lemma}
\begin{proof}
By Lemma \ref{l:numberBranch} the range of the integer $M$ of Definition \ref{d:topo} among all $T\in\OTP(b)$ is finite. In turn this implies that the set of possible topologies of currents $T\in\OTP(b)$ is finite. Indeed the topology of a polyhedral current $T$ as in Definition \ref{d:topo}, up to choosing the order of the points $\{x_1,\dots, x_M\}$, is uniquely determined by the $M\times M$ matrix $A:=(|{\rm{sign}}(a_{ij})|)_{ij}$. Hence it is sufficient to prove that if $T$ and $T'$ are in $\OTP(b)$ and have the same topology, then $T=T'$, and by Lemma \ref{l:supports_determine_current} it suffices to prove that $\supp(T)=\supp(T')$.

By \cite[Proposition 7.8]{BCM} the support of $T$ and $T'$ does not contain loops, hence we can apply Lemma \ref{l:topo_and_support} and we can assume that $T$ and $T'$ can be written as in \eqref{e:def_topo} with $a_{ij}=a'_{ij}$ for every $i,j=1,\dots, M$. This means that the set of competitors for the branched transportation problem with boundary $b$ and a given topology can be reduced to a family of polyhedral currents $T\in\Po_1(K)$ whose only unknown is the position of the points $\{x_1,\dots,x_M\}\setminus \supp(b)$. Accordingly, we denote $n:=\Haus^0(\supp(b))$ and we order the points $\{x_1,\dots, x_M\}$ in such a way that $\BR(T)=\{x_1,\dots,x_{M-n}\}$.
The $\alpha$-mass of such $T$ is computed as 
$$\MM(T)=\sum_{i<j}|a_{ij}|^\alpha\Haus^1(\sigma_{ij})$$
and by the previous discussion, since the vector $(x_{M-n+1},\dots, x_M)$ is fixed, this is a functional of the vector $(x_1,\dots,x_{M-n})$ only, which can be written as
\begin{equation}\label{e:convex}
    \MM(T)=F(x_1,\dots,x_{M-n}):=\sum_{i<j}|a_{ij}|^\alpha|x_j-x_i|=C+\sum_{i=1}^{M-n}\sum_{j=i+1}^M |a_{ij}|^\alpha|x_j-x_i|,
\end{equation}
where $C= \sum_{i=M-n+1}^{M}\sum_{i<j}|a_{ij}|^\alpha|x_j-x_i|$.
One can immediately see that $F$ is convex, being a sum of convex functions. Moreover each term $|a_{ij}|^\alpha|x_j-x_i|$ in \eqref{e:convex}, as a function of the variable $x_j$, is strictly convex on a segment $[s,t]$ whenever $x_i$, $s$ and $t$ are not collinear. 

Assume by contradiction that $T\neq T'\in\OTP(b)$ have the same topology and consider the corresponding sets $$\BR(T)=\{x_1,\dots,x_{M-n}\},\quad \BR(T')=\{x'_1,\dots,x'_{M-n}\}.$$
By Lemma \ref{l:topo_and_support} there exists $j\in\{1,\dots,M-n\}$ such that $x_j\neq x'_j$. As in the proof of Lemma \ref{l:numberBranch} we infer that $x_j$ is an endpoint of at least three segments in the support of $T$ which are not collinear. We deduce by the discussion after \eqref{e:convex} that the function $F$ is strictly convex in the $j$-th variable. Since $F(x_1,\dots,x_{M-n})=F(x'_1,\dots,x'_{M-n})$ we deduce that there exists a point $(y_1,\dots,y_{M-n})$ with \begin{equation}\label{e:bettery}
    F(y_1,\dots,y_{M-n})<F(x_1,\dots,x_{M-n}).
\end{equation}
Denote 
$$
z_{i}:=
\begin{cases}
y_i \quad \mbox{if $i\leq M-n$}\\
x_i \quad \mbox{otherwise}
\end{cases}
$$
and let $S$ be the current
\begin{equation}\label{e:S}
    S:=\sum_{i<j}a_{ij}\llbracket\tilde\sigma_{ij}\rrbracket,
\end{equation}
where $\tilde\sigma_{ij}$ is the segment with first endpoint $z_i$ and second endpoint $z_j$. Notice that in principle it might happen that $S$ does not have the same topology as $T$ and $T'$, since \eqref{e:S} might fail to have property (iii) of Definition \ref{d:topo}. However we have $\partial S=b$ and by \eqref{e:bettery} 
$$\MM(S)\leq F(y_1,\dots,y_{M-n})<F(x_1,\dots,x_{M-n})=\MM(T),$$
which contradicts the assumption $T\in\OTP(b)$. 
\end{proof}


\begin{lemma}\label{l:magic_points}
    For every boundary $b\in\I_{0}(K)$ and $T \in \OTP(b)$, there is a set of distinct points $\{p_1, \dots, p_h\} \subset \supp(T)\setminus(\BR(T)\cup\supp (b))$ such that 
    \[ \big\{ S \in \OTP(b): \{ p_1, \dots, p_h \} \subset \supp(S)\big\} = \{T\}\,. \]
    Moreover, the $p_i$'s can be chosen so that if $p_i\in\supp(T)\cap\supp(S)$ for some $ S\in\OTP(b)$, then there exists $\rho>0$ such that $\supp(T)\cap B_\rho(p_i)=\supp(S)\cap B_\rho(p_i)$.
\end{lemma}

\begin{proof}
By Lemma \ref{l:finiteminimizers} we have that $\OTP(b)$ consists of finitely many polyhedral currents $ T^1,\dots, T^h$ and, by Lemma \ref{l:supports_determine_current}, the symmetric difference $\supp(T^i)\triangle \,\supp(T^j)$ is a relatively open set of positive length for every $i\neq j$. Up to reordering, we assume $T^1=T$ and for every $i \in \{2, \dots, h\}$ we consider the set $U_i:= \supp(T) \setminus (\supp(T^i) \cup \BR(T))$. We observe, recalling that $\BR(T)$ is finite by Lemma \ref{l:numberBranch}, that each $U_i$ is relatively open with positive length. 
Define the subset
\begin{align*}
    V_i := U_i \cap  \Big( \bigcup_{j \neq i} \BR(T^j) \cup \left\{p \in U_i: \text{$\supp(T^j)$ intersects $U_i$ transversally at $p$} \right\} \Big)
\end{align*}
and observe that $V_i$ is finite since every $T^j$ is polyhedral. Then choose $p_i \in U_i \setminus V_i$. Clearly $p_i\in \supp(T)\setminus(\BR(T)\cup \supp(b))$ and $p_i\not\in\supp(T^i)$; moreover if $p_i \in \supp(T^j)$ then locally $\supp(T^j)$ agrees with $\supp(T)$.
\end{proof}

\section{Density of boundaries with unique minimizers: perturbation argument}
\subsection{Construction of the perturbed boundaries}
Let us fix a boundary $b\in\I_0(K)$, an integer polyhedral current $T\in\OTP(b)$, see Lemma \ref{l:integral_minimizers}, and points $\{p_1,\dots,p_h\}$ as in Lemma \ref{l:magic_points}. For a fixed $k\in\N\setminus\{0\}$ and for $n=1,2,\dots$ we denote 
\begin{equation}\label{e:bienne}
\begin{split}
    T_n &:=T- \frac1k \sum_{i=1}^h T\trace B_{n^{-1}}(p_i)\, ,\\
    b_n &:=\partial T_n\,.
\end{split}
\end{equation}
Observe that by \cite[Proposition 3.6]{CDRMcalcvar}, the multiplicity of $T$ is bounded from above by $2^{-1}\Mass(b)$ and moreover, for $n$ sufficiently large, the closed balls $\overline B_{n^{-1}}(p_i)$ are disjoint and do not intersect $\supp(b)\cup\BR(T)$, so that we have
\begin{equation}\label{e:massa_bienne}
    \Mass(b_n)=\Mass(b)+k^{-1}\sum_{i=1}^h\Mass(\partial(T\trace B_{n^{-1}}(p_i)))\leq\Mass(b)+hk^{-1}\Mass(b)
\end{equation}
and \begin{equation}\label{e:convergence_bienne}
    \Flat_K(b_n-b)\leq k^{-1}\sum_{i=1}^h\Mass(T\trace B_{n^{-1}}(p_i))\leq h(nk)^{-1}\Mass(b).
\end{equation}
For every $n$, we choose $S_n\in\OTP(b_n)$ and we apply Lemma \ref{l:integral_minimizers} to the boundaries $kb_n$ to deduce that $kS_n\in\Po_1(K)\cap\I_1(K)$. By \eqref{e:bienne} we have 
\begin{equation}\label{e:alphamass_tienne}
\MM(S_n)\leq\MM(T_n)<\MM(T). \end{equation}
The aim of this section is to prove the following:
\begin{proposition}\label{p:unique_bienne}
There exists $k_0=k_0(\alpha)$ such that for $(b_n)_{n}$ as in \eqref{e:bienne} with $k\geq k_0$ and for $n$ sufficiently large, $\OTP(b_n)=\{T_n\}$. 
\end{proposition}

In the next lemma, for any set $A$ and $\rho>0$ we denote $B_\rho(A):=\bigcup_{a\in A}B_\rho(a)$.
\begin{lemma}\label{l:conv_hauss}
    For $n \in \N \setminus \{0\}$ let $b_n$ be as in \eqref{e:bienne} and $S_n\in\OTP(b_n)$. For every subsequence $(S_{n_j})_{j\in\N}$ and current $S$ such that $\Flat_K(S_{n_j}-S)\to 0$ as $j\to\infty$ we have $S\in\OTP(b)$ and moreover for every $\rho>0$ we have $\supp(S_{n_j})\subset B_\rho(\supp(S)\cup\{p_1,\dots,p_h\})$, for $j$ sufficiently large.
\end{lemma}
\begin{proof}
The first part of the proposition is a direct consequence of Theorem \ref{t:stability_new}. 
Towards a proof by contradiction of the second part, assume that there exists $r>0$ and, for every $j$, a point 
\begin{equation}\label{e:quenne}
    q_j\in\supp(S_{n_j})\setminus B_{2r}(\supp(S)\cup\{p_1,\dots,p_h\}).
\end{equation} By \cite[Proposition 2.6]{CDRMS} we have 
\begin{equation}\label{e:mass_in}
    \liminf_j\MM(S_{n_j}\trace B_{r}(\supp(S)\cup\{p_1,\dots,p_h\}))\geq\MM(S).
\end{equation}
On the other hand, by \eqref{e:bienne} and \eqref{e:quenne} the current $S_{n_j}$ has no boundary in $B_r(q_j)$, for $j$ sufficiently large. Moreover, by \cite[Proposition 7.4]{BCM} the restriction $R_j$ of $S_{n_j}\trace \,B_r(q_j)$ to the connected component of its support containing $q_j$ has non-trivial boundary, and more precisely applying \cite[Lemma 28.5]{SimonLN} with $f(x)=|x-q_j|$ we deduce that $\emptyset\neq\supp(\partial R_j)\subset \partial B_r(q_j)$. We conclude that $\supp (R_j)$ contains a path connecting $q_j$ to a point of $\partial B_r(q_j)$. By Lemma \ref{l:integral_minimizers} such path has multiplicity bounded from below by $k^{-1}$. This allows to conclude that 
\begin{equation}\label{e:mass_out}
\MM(S_{n_j}\trace B_r(q_j))\geq rk^{-\alpha}    
\end{equation}
Combining \eqref{e:quenne},\eqref{e:mass_in}, and \eqref{e:mass_out}, we conclude
$$ \liminf_j\MM(S_{n_j})\geq\MM(S)+rk^{-\alpha}=\MM(T)+rk^{-\alpha},$$
which contradicts \eqref{e:alphamass_tienne}.
\end{proof}

We dedicate the rest of this section to prove the following:

\begin{lemma}\label{l:main}
    There exists $k_0=k_0(\alpha)$ with the following property. Let $T$ and $(T_n)_n$ be as in \eqref{e:bienne} with $k\geq k_0$ and let $S$ and $(S_{n_j})_j$ be as in Lemma \ref{l:conv_hauss}.  Then $S_{n_j}=T_{n_j}$, for $j$  sufficiently large and in particular  $S=T$.
\end{lemma}

\begin{proof}
We divide the proof in three steps. In \S \ref{Step1} we prove that locally in a box around each point $p\in\{p_1,\dots,p_h\}\cap \supp(S)$ for $j$ sufficiently large the restriction to the box of the current $S_{n_j}$ is a minimizer of the $\alpha$-mass for a certain boundary whose support is a set of four \emph{almost collinear} points. In \S\ref{Step2} we analyze all the possible topologies for the minimizers with such boundary and we are able to exclude all of them except for two. In \S \ref{Step3} we combine the local analysis with a global energy estimate to conclude.\\ 


\subsection{Local structure of $S_{n_j}$.}\label{Step1} Let $\rho$ be sufficiently small, to be chosen later (see (2a), (3a) and (3b) in \S\ref{Step2}). For every $i=1,\dots, h$ and for $p_i\in\supp(S)$, by Lemma \ref{l:magic_points} we can choose orthonormal coordinates $(x,y)\in\R\times\R^{d-1}$ such that, up to a dilation with homothety ratio $c$ with $$c>\frac{8}{\text{dist}(p_i,\supp(b)\cup\BR(S))},$$ denoting $Q:=[-8,8]\times B_{\rho}^{d-1}(0)$ and $B_j:=(-cn_j^{-1},0)$, $C_j:=(cn_j^{-1},0)$, for $j=1,2,\dots$, it holds:  
\begin{itemize}
    \item [(i)] $p_i=(0,0)$;
    \item [(ii)]$S\trace Q=\theta\llbracket\sigma\rrbracket$, where $\theta\in\Z$ and $\sigma:=[-8,8]\times\{0\}^{d-1}$ is positively oriented;
    \item [(iii)] $\{p_1,\dots,p_h\}\cap \sigma= \{p_i\}$;
    \item [(iv)] for $j$ sufficiently large it holds $b_{n_j}\trace Q =k^{-1}\theta(\delta_{B_j}-\delta_{C_j})$.
\end{itemize} 

\begin{figure}[ht]
	\centering
	\captionsetup{justification=centering}
	\includegraphics[width=0.6\textwidth]{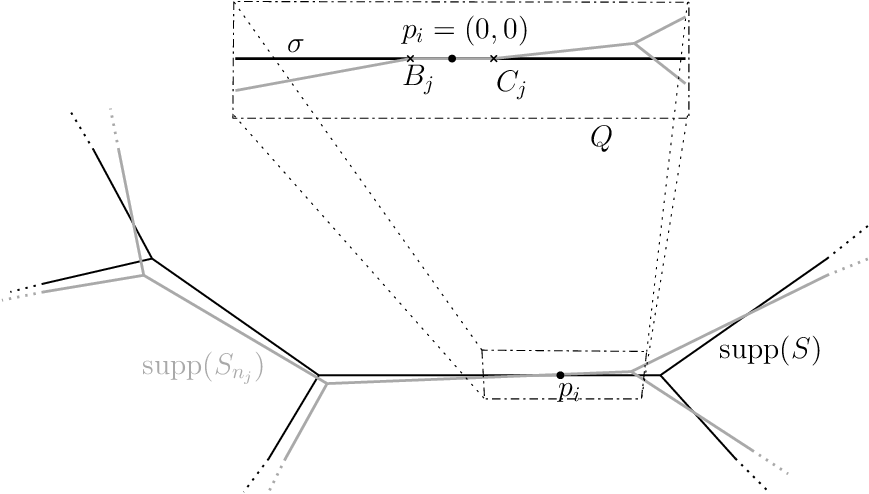}
	\caption{Illustration of our choice of $Q$.}\label{f:local_structure}
\end{figure}

The choice of $Q$ is illustrated in Figure \ref{f:local_structure}. By Lemma \ref{l:conv_hauss}, for this $\rho$, we may choose $j$ large enough such that
\begin{equation}\label{e:where_is_Sn}
    \supp(S_{n_j})\cap \bar{Q} \subset B_\rho(\sigma).
\end{equation}
For $x \in \R$ we denoting by $S_{n_j}^x$ the slice of $S_{n_j}\trace Q$ at the point $x$
with respect to the projection $\pi:\R\times\R^{d-1}\to\R$, see \cite[\S 28]{SimonLN} or \cite[\S 4.3]{FedererBOOK}. We infer from the flat convergence of $S_{n_j}$ to $S$ that for $\Haus^1$-a.e. $x\in[-8,8]$ it holds
\begin{equation}\label{e:flat_slices}
    \Flat_K(S_{n_j}^x-\theta\delta_{(x,0)})\to 0\quad \mbox{as $j\to \infty$}
\end{equation}
and moreover by Lemma \ref{l:integral_minimizers} the multiplicities of $S_{n_j}^x$ are integer multiples of $k^{-1}$. 

We aim to prove that for $j$ sufficiently large there are points $y_j^{\pm}\in B^{d-1}_\rho(0)$ such that \begin{equation}\label{e:good_slices}
S_{n_j}^{\pm 4}=\theta\delta_{(\pm 4,y_j^{\pm})},    
\end{equation}

To this aim we seek points $x_1(j)\in[-6,-5]$, $x_2(j)\in[-2,-1]$,  $x_3(j)\in[1,2]$ and
$x_4(j)\in[5,6]$ such that for $i=1,\dots,4$ it holds
\begin{equation}\label{e:almost_good_slices}
    S_{n_j}^{x_i(j)}=\theta\delta_{(x_i(j),y_i(j))},
\end{equation}
for some points $y_i(j)\in B^{d-1}_\rho(0)$. If so, by \cite[Lemma 28.5]{SimonLN}, 
\eqref{e:almost_good_slices} and \eqref{e:where_is_Sn} imply, denoting $$Q_1^j:= (x_1(j),x_2(j)) \times B_{\rho}^{d-1}(0)\quad\mbox{and}\quad Q_2^j:=(x_3(j),x_4(j)) \times B_{\rho}^{d-1}(0),$$
that $$\partial(S_{n_j}\trace Q_1^j)=S_{n_j}^{x_2(j)}-S_{n_j}^{x_1(j)}\quad \mbox{and}\quad\partial(S_{n_j}\trace Q_2^j)=S_{n_j}^{x_4(j)}-S_{n_j}^{x_3(j)}.$$
In turn, by \cite[Proposition 7.4]{BCM}
the latter implies \eqref{e:good_slices}.

In order to prove \eqref{e:almost_good_slices}, we focus on the interval $I:=[1,2]$ as the argument for the remaining intervals is identical. Firstly, we observe that by \eqref{e:flat_slices} we have
\begin{equation}\label{e:mass_slices_larger_theta}
\liminf_j(\Mass(S_{n_j}^x))\geq \theta \quad\mbox{for $\Haus^1$-a.e. $x\in I$}.
\end{equation}
Next, denoting $\Omega:=I\times B_{\rho}^{d-1}(0)$, we claim that for $j$ sufficiently large and for every $C>0$ it holds that
\begin{equation}\label{e:slices_with_small_energy}
    \Haus^1(\{x\in I:\MM(S_{n_j}^x)\leq \theta^\alpha + C\})>0,
\end{equation}
where for a 0-current $Z:=\sum_{\ell\in\N}\theta_\ell\delta_{z_\ell}$ we denoted $\MM(Z):=\sum_{\ell\in\N}|\theta_\ell|^\alpha$.

Assume by contradiction that \eqref{e:slices_with_small_energy} is false for infinitely many indices $j$. By \cite[equation (3.11)]{CDRMcpam}, for those indices we have
\begin{equation}\label{e:towards_contr}
    \MM(S_{n_j}\trace \bar \Omega)\geq\theta^\alpha+C=\MM(S\trace \bar \Omega)+C.
\end{equation}
The latter, combined with \eqref{e:alphamass_tienne}, implies that for the same indices we have 
$$\MM(S_{n_j}\trace (\R^d\setminus \bar \Omega))<\MM(S\trace (\R^d\setminus \bar \Omega))-C,$$
which contradicts \cite[Proposition 2.6]{CDRMS}. 
From \eqref{e:mass_slices_larger_theta} and \eqref{e:slices_with_small_energy} we deduce that for $j$ sufficiently large there exists $x_1(j)\in I$ such that
\begin{equation}\label{e:final_good_slice}
\Mass(S_{n_j}^{x_1(j)})\geq\theta\quad\mbox{and}\quad \MM(S_{n_j}^{x_1(j)})\leq\theta^\alpha+C.
\end{equation}
Lastly we prove that if $C$ is sufficiently small, then \eqref{e:final_good_slice} implies 
\begin{equation}\label{e:GS}
    S_{n_j}^{x_1(j)}=\theta\delta_{(x_1(j),y_1(j))},
\end{equation}
for some points $y_1(j)\in B^{d-1}_\rho(0)$, thus completing the proof of \eqref{e:almost_good_slices}.

Towards a proof by contradiction of \eqref{e:GS}, observe that for every 0-current $Z=\sum_{\ell=1}^M\theta_\ell\delta_{z_\ell}$, with $M\geq 2$, $|\theta_\ell|\geq k^{-1}$ and $z_\ell$ distinct, satisfying $\Mass(Z)=\sum_\ell|\theta_\ell|\geq \theta$, the strict subadditivity of the function $t\mapsto t^\alpha$ (for $t>0$) yields the existence of a $\bar C=\bar C(\alpha, \theta, k)> 0$ such that
$$\MM(Z)=|\theta_1|^\alpha+\sum_{\ell=2}^{M}|\theta_\ell|^\alpha\geq \min\{(mk^{-1})^\alpha + (\theta-mk^{-1})^\alpha: m=1,\dots,k\theta-1\}>\theta^\alpha+\bar C.$$
This contradicts \eqref{e:final_good_slice}, by the arbitrariness of $C$.\\

It follows from \eqref{e:good_slices} and \cite[Lemma 28.5]{SimonLN} that, denoting $$Q':=(-4,4) \times B_{\rho}^{d-1}(0), \quad A_j:= ( -4,y_j^-),\quad D_j:=(4,y_j^+),$$
we have
\begin{equation}\label{e:four_points}
\partial (S_{n_j}\trace Q') =\theta\left(\delta_{D_j}-\delta_{A_j}+k^{-1}\big(\delta_{B_j}-\delta_{C_j}\big)\right),   
\end{equation}
for $j$ sufficiently large (see Figure \ref{f:wurst}).

\begin{figure}[ht]
	\centering
	\captionsetup{justification=centering}
	\includegraphics[width=0.9\textwidth]{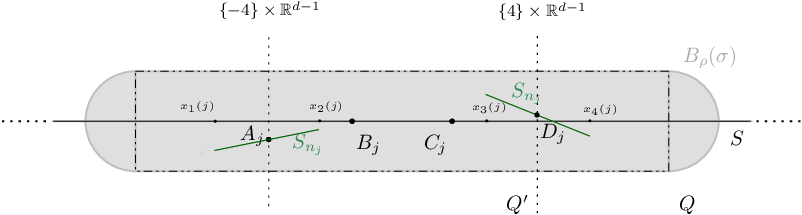}
	\caption{Representation of parts of $S_{n_j}\trace Q$.}\label{f:wurst}
\end{figure}
\vspace{0.5cm}

\subsection{Analysis of the possible topologies of $S_{n_j}\trace Q'$.}\label{Step2}
Since $S_{n_j}\in\OTP(b_{n_j})$ we must have $S_{n_j}\trace Q'\in\OTP(\partial (S_{n_j}\trace Q'))$. In general, we will denote by $\sigma_{PR}$ the oriented segment from the point $P$ to the point $R$. We aim to prove that for $k\geq k_0(\alpha)$ and for $\rho \leq \rho(k)$ sufficiently small it holds $S_{n_j}\trace Q' \in \{W_j, Z_j\}$, for $j$ large enough, where 
\begin{equation}\label{e:THE_ONE}
    W_j := \theta \left( \llbracket \sigma_{A_jB_j} \rrbracket + \llbracket \sigma_{C_jD_j} \rrbracket + \frac{k-1}{k} \llbracket \sigma_{B_jC_j} \rrbracket \right)\quad \mbox{and}\quad Z_j := \theta \left( \llbracket \sigma_{A_jD_j} \rrbracket + \frac{1}{k} \llbracket \sigma_{C_jB_j} \rrbracket \right) \,, 
\end{equation}
see Table \ref{FigONES}.
We will do this by excluding every other topology comparing angle conditions which are given by the multiplicities of the segments (which depend on $k$) and contradict the choice of $\rho$. Thus, when we say for $\rho$ small enough, we mean implicitly to choose $j$ large enough such that by Lemma \ref{l:conv_hauss}, we have $\supp(S_{n_j})\subset B_\rho(\supp(S)\cup\{p_1,\dots,p_h\})$ for the desired $\rho$ .

Write $S_{n_j}\trace Q'=\sum_{i<j}a_{ij}\llbracket\sigma_{ij}\rrbracket$ as in \eqref{e:def_topo} and observe that by Lemma \ref{l:numberBranch}, as $\Haus^0(\partial (S_{n_j}\trace Q'))=4$, then $\Haus^0(\BR (S_{n_j}\trace Q'))\in \{0,1,2\}$. We thus analyze the three cases separately and we recall that, by Lemma \ref{l:supports_determine_current}, in order to prove \eqref{e:THE_ONE} it suffices to prove that $$\supp(S_{n_j}\trace Q')=\sigma_{A_jB_j}\cup\sigma_{B_jC_j}\cup\sigma_{C_jD_j}\quad\mbox{or}\quad \supp(S_{n_j}\trace Q')=\sigma_{A_jD_j}\cup\sigma_{B_jC_j}.$$

\begin{longtable}{  p{1cm} c } \label{FigONES} 

\\ $W_j$ &
    \begin{tikzpicture}
         \filldraw[black] (0,0) circle (2pt) node[anchor=south]{$A$};
        \filldraw[black] (3.6, 2) circle (2pt) node[anchor=north]{$B$};
        \filldraw[black] (5.4, 2) circle (2pt) node[anchor=north]{$C$};
        \filldraw[black] (8.8, 2.8) circle (2pt) node[anchor=north]{$D$};
        \path[draw = gray,% 
             decoration={%
            markings,%
            mark=at position 0.5   with {\arrow[scale=1.5]{>}},%
            },%
            postaction=decorate] (0,0) -- (3.6, 2)
            node[midway, above, color= gray, font=\small]{$\theta$};
        \path[draw = gray,% 
             decoration={%
            markings,%
            mark=at position 0.5   with {\arrow[scale=1.5]{>}},%
            },%
            postaction=decorate] (5.4, 2) -- (8.8, 2.8)
            node[midway, below, color= gray, font=\small]{$\theta$};
        \path[draw = gray,% 
             decoration={%
            markings,%
            mark=at position 0.5   with {\arrow[scale=1.5]{>}},%
            },%
            postaction=decorate] (3.6, 2) -- (5.4, 2)
            node[midway, above, color= gray, font=\small]{$\theta-\delta$};
    \end{tikzpicture}
    \\ \hline
\\ $Z_j$ & 
\begin{tikzpicture}
     \filldraw[black] (0,0) circle (2pt) node[anchor=south]{$A$};
    \filldraw[black] (3.6, 2) circle (2pt) node[anchor=south]{$B$};
    \filldraw[black] (5.4, 2) circle (2pt) node[anchor=south]{$C$};
    \filldraw[black] (8.8, 2.8) circle (2pt) node[anchor=north]{$D$};
    \path[draw = gray,% 
             decoration={%
            markings,%
            mark=at position 0.5   with {\arrow[scale=1.5]{>}},%
            },%
            postaction=decorate] (0,0) -- (8.8, 2.8)
            node[midway, below, color= gray, font=\small]{$\theta$};
    \path[draw = gray,% 
             decoration={%
            markings,%
            mark=at position 0.5   with {\arrow[scale=1.5]{<}},%
            },%
            postaction=decorate] (3.6, 2) -- (5.4, 2)
            node[midway, above, color= gray, font=\small]{$\delta$};
\end{tikzpicture}
\\ \hline
\caption{Representation of $W_j$ and $Z_j$. From now on $\delta := \theta/k$ and we remove the subscript $j$ from the points.}
\end{longtable}

\emph{Case 1}: $\BR(S_{n_j}\trace Q') = \emptyset$. Recalling \cite[Proposition 7.4]{BCM},  $\supp(S_{n_j}\trace Q')$ must be one of the following sets, sorted alphabetically: 
\begin{itemize}
    \item [(1a)] $\sigma_{A_jB_j}\cup\sigma_{A_jC_j}\cup\sigma_{A_jD_j}$,
    \item [(1b)] $\sigma_{A_jB_j}\cup\sigma_{A_jC_j}\cup\sigma_{B_jD_j}$,
    \item [(1c)] $\sigma_{A_jB_j}\cup\sigma_{A_jC_j}\cup\sigma_{C_jD_j}$,
    \item [(1d)] $\sigma_{A_jB_j}\cup\sigma_{A_jD_j}\cup\sigma_{B_jC_j}$,
     \item [(1e)] $\sigma_{A_jB_j}\cup\sigma_{A_jD_j}\cup\sigma_{C_jD_j}$,
    \item [(1f)] $\sigma_{A_jB_j}\cup\sigma_{B_jC_j}\cup\sigma_{B_jD_j}$,
    \item [(1g)] $\sigma_{A_jB_j}\cup\sigma_{B_jC_j}\cup\sigma_{C_jD_j}$,
    \item [(1h)] $\sigma_{A_jB_j}\cup\sigma_{B_jD_j}\cup\sigma_{C_jD_j}$,
    \item [(1i)]
    $\sigma_{A_jB_j}\cup\sigma_{C_jD_j}$,
    \item [(1j)]
    $\sigma_{A_jC_j}\cup\sigma_{A_jD_j}\cup\sigma_{B_jC_j}$,
    \item [(1k)]
    $\sigma_{A_jC_j}\cup\sigma_{A_jD_j}\cup\sigma_{B_jD_j}$,    
    \item [(1l)] $\sigma_{A_jC_j}\cup\sigma_{B_jC_j}\cup\sigma_{B_jD_j}$,
    \item [(1m)] $\sigma_{A_jC_j}\cup\sigma_{B_jC_j}\cup\sigma_{C_jD_j}$,    
    \item [(1n)] $\sigma_{A_jC_j}\cup\sigma_{B_jD_j}$,
    \item [(1o)] $\sigma_{A_jC_j}\cup\sigma_{B_jD_j}\cup\sigma_{C_jD_j}$,    
    \item [(1p)] $\sigma_{A_jD_j}\cup\sigma_{B_jC_j}$,    
    \item [(1q)] $\sigma_{A_jD_j}\cup\sigma_{B_jC_j}\cup\sigma_{B_jD_j}$,
    \item [(1r)] $\sigma_{A_jD_j}\cup\sigma_{B_jC_j}\cup\sigma_{C_jD_j}$,
    \item [(1s)] $\sigma_{A_jD_j}\cup\sigma_{B_jD_j}\cup\sigma_{C_jD_j}$.
\end{itemize}
Observe that we omitted the cases
\begin{itemize}
    \item [(i)] $\sigma_{A_jB_j}\cup\sigma_{A_jC_j}\cup\sigma_{B_jC_j}$,
    \item [(ii)] $\sigma_{A_jB_j}\cup\sigma_{A_jD_j}\cup\sigma_{B_jD_j}$,
    \item [(iii)] $\sigma_{A_jC_j}\cup\sigma_{A_jD_j}\cup\sigma_{C_jD_j}$ 
    \item [(iv)] $\sigma_{B_jC_j}\cup\sigma_{B_jD_j}\cup\sigma_{C_jD_j}$
\end{itemize}
because, independently of the position of the points, the support either contains a loop or does not contain one of the four points in the support of the boundary. The only exceptions to this behaviour are (ii) and (iii) only when the four points are collinear, which is not relevant, as we discuss in Sub-case 1-1 below.\\

\emph{Sub-case 1-1}. Firstly we observe that when the points $A_j,B_j,C_j,D_j$ are collinear the only admissible competitor is $Z_j$.\\ 

\emph{Sub-case 1-2}. Next, we analyze the case in which no triples among the points $A_j,B_j,C_j,D_j$ are contained in a line.

We immediately exclude those cases for which the corresponding set is not the support of any current with boundary $\partial(S_{n_j}\trace Q')$. Hence we can exclude (1i) and (1n), because the endpoints of the two segments in the support have different multiplicities. Moreover we exclude (1d), (1j), (1q) and (1r) as well, because the segment $\sigma_{A_j,D_j}$ should have multiplicity $\theta$, being either for $A_j$ or $D_j$ the only segment in the support containing it. On the other hand, the remaining point (respectively $D_j$ or $A_j$) is an endpoint also for a different segment of the support, from which we deduce that the multiplicity of the latter segment should be 0 (see Table \ref{Fig1}).

\begin{longtable}{  p{1cm} c } \label{Fig1} 
\\ 1d &
\begin{tikzpicture}
     \filldraw[black] (0,0) circle (2pt) node[anchor=south]{$A$};
    \filldraw[black] (1.8, 1) circle (2pt) node[anchor=south]{$B$};
    \filldraw[black] (2.6, 1) circle (2pt) node[anchor=south]{$C$};
    \filldraw[black] (4.4, 1.4) circle (2pt) node[anchor=north]{$D$};
    \draw[gray, thick] (0,0) -- (1.8, 1);
    \draw[gray, thick] (0,0) -- (4.4, 1.4);
    \draw[gray, thick] (1.8, 1) -- (2.6, 1);
\end{tikzpicture}
  \qquad \qquad
\begin{tikzpicture}
     \filldraw[black] (0,0) circle (2pt) node[anchor=south]{$A$};
    \filldraw[black] (1.8, 1) circle (2pt) node[anchor=south]{$B$};
    \filldraw[black] (2.6, 1) circle (2pt) node[anchor=south]{$C$};
    \filldraw[black] (4.4, 1.4) circle (2pt) node[anchor=north]{$D$};
    \draw[gray, thick] (0,0) -- (1.8, 1);
    \draw[gray, thick] (2.6, 1) -- (4.4, 1.4);
\end{tikzpicture}
\quad  1i 
\\ \hline 
\\ 1j &
\begin{tikzpicture}
     \filldraw[black] (0,0) circle (2pt) node[anchor=south]{$A$};
    \filldraw[black] (1.8, 1) circle (2pt) node[anchor=south]{$B$};
    \filldraw[black] (2.6, 1) circle (2pt) node[anchor=south]{$C$};
    \filldraw[black] (4.4, 1.4) circle (2pt) node[anchor=north]{$D$};
    \draw[gray, thick] (0,0) -- (2.6, 1);
    \draw[gray, thick] (1.8, 1) -- (2.6, 1);
    \draw[gray, thick] (0,0) -- (4.4, 1.4);
\end{tikzpicture}
 \qquad \qquad
\begin{tikzpicture}
     \filldraw[black] (0,0) circle (2pt) node[anchor=south]{$A$};
    \filldraw[black] (1.8, 1) circle (2pt) node[anchor=south]{$B$};
    \filldraw[black] (2.6, 1) circle (2pt) node[anchor=north]{$C$};
    \filldraw[black] (4.4, 1.4) circle (2pt) node[anchor=north]{$D$};
    \draw[gray, thick] (0,0) -- (2.6, 1);
    \draw[gray, thick] (1.8, 1) -- (4.4, 1.4);
\end{tikzpicture}
\quad 1n
\\ \hline 
\\ 1q &
\begin{tikzpicture}
     \filldraw[black] (0,0) circle (2pt) node[anchor=south]{$A$};
    \filldraw[black] (1.8, 1) circle (2pt) node[anchor=south]{$B$};
    \filldraw[black] (2.6, 1) circle (2pt) node[anchor=south]{$C$};
    \filldraw[black] (4.4, 1.4) circle (2pt) node[anchor=north]{$D$};
    \draw[gray, thick] (0,0) -- (4.4, 1.4);
    \draw[gray, thick] (1.8, 1) -- (2.6, 1);
    \draw[gray, thick] (1.8, 1) -- (4.4, 1.4);
\end{tikzpicture}
 \qquad \qquad
\begin{tikzpicture}
     \filldraw[black] (0,0) circle (2pt) node[anchor=south]{$A$};
    \filldraw[black] (1.8, 1) circle (2pt) node[anchor=south]{$B$};
    \filldraw[black] (2.6, 1) circle (2pt) node[anchor=south]{$C$};
    \filldraw[black] (4.4, 1.4) circle (2pt) node[anchor=north]{$D$};
    \draw[gray, thick] (0,0) -- (4.4, 1.4);
    \draw[gray, thick] (1.8, 1) -- (2.6, 1);
    \draw[gray, thick] (2.6, 1) -- (4.4, 1.4);
\end{tikzpicture}
\quad 1r
\\ \hline
\caption{Representation of (1d), (1i), (1j), (1n), (1q), (1r).}
\end{longtable}

We exclude the following cases by direct comparison with the $\alpha$-mass of $Z_j$, for $j$ sufficiently large (see Table \ref{Fig2}):

\begin{itemize}
    \item (1a), whose corresponding $\alpha$-mass is
    $$\theta^\alpha(\Haus^1(\sigma_{A_jD_j})+k^{-\alpha}(\Haus^1(\sigma_{A_jB_j})+\Haus^1(\sigma_{A_jC_j}))>\MM(Z_j).$$
    \item (1b), whose corresponding $\alpha$-mass is
    $$\theta^\alpha((1+k^{-1})^{\alpha}\Haus^1(\sigma_{A_jB_j})+\Haus^1(\sigma_{B_jD_j})+k^{-\alpha}\Haus^1(\sigma_{A_jC_j}))>\MM(Z_j).$$
    \item (1f), whose corresponding $\alpha$-mass is
    $$\theta^\alpha(\Haus^1(\sigma_{A_jB_j})+\Haus^1(\sigma_{B_jD_j})+k^{-\alpha}\Haus^1(\sigma_{B_jC_j}))>\MM(Z_j).$$
     \item (1k), whose corresponding $\alpha$-mass is
    $$\theta^\alpha((1+k^{-1})^{\alpha}\Haus^1(\sigma_{A_jD_j})+k^{-\alpha}(\Haus^1(\sigma_{A_jC_j})+\Haus^1(\sigma_{B_jD_j})))>\MM(Z_j).$$
     \item (1l), whose corresponding $\alpha$-mass is
    $$\theta^\alpha((1+k^{-1})^{\alpha}\Haus^1(\sigma_{B_jC_j})+\Haus^1(\sigma_{A_jC_j})+\Haus^1(\sigma_{B_jD_j}))>\MM(Z_j).$$
     \item (1m), whose corresponding $\alpha$-mass is
    $$\theta^\alpha(\Haus^1(\sigma_{A_jC_j})+\Haus^1(\sigma_{C_jD_j})+k^{-\alpha}\Haus^1(\sigma_{B_jC_j}))>\MM(Z_j).$$
    \item (1o), whose corresponding $\alpha$-mass is
    $$\theta^\alpha((1+k^{-1})^{\alpha}\Haus^1(\sigma_{C_jD_j})+\Haus^1(\sigma_{A_jC_j})+k^{-\alpha}\Haus^1(\sigma_{B_jD_j}))>\MM(Z_j).$$
    \item (1s), whose corresponding $\alpha$-mass is
    $$\theta^\alpha(\Haus^1(\sigma_{A_jD_j})+k^{-\alpha}(\Haus^1(\sigma_{B_jD_j})+\Haus^1(\sigma_{C_jD_j}))>\MM(Z_j).$$
\end{itemize}

\begin{longtable}{  p{1cm} c } \label{Fig2} 
        1a &
        \begin{tikzpicture}
         \filldraw[black] (0,0) circle (2pt) node[anchor=south]{$A$};
        \filldraw[black] (3.6, 2) circle (2pt) node[anchor=south]{$B$};
        \filldraw[black] (5.4, 2) circle (2pt) node[anchor=south]{$C$};
        \filldraw[black] (8.8, 2.8) circle (2pt) node[anchor=north]{$D$};
        \path[draw = gray,% 
             decoration={%
            markings,%
            mark=at position 0.5   with {\arrow[scale=1.5]{>}},%
            },%
            postaction=decorate] (0,0) -- (3.6, 2)
            node[midway, above left, color= gray, font=\small]{$\delta$};
        \path[draw = gray,% 
             decoration={%
            markings,%
            mark=at position 0.5   with {\arrow[scale=1.5,]{<}},%
            },%
            postaction=decorate] (0,0) -- (5.4, 2)
            node[midway, above, color= gray, font=\small]{$\delta$};
        \path[draw = gray,% 
             decoration={%
            markings,%
            mark=at position 0.5   with {\arrow[scale=1.5]{>}},%
            },%
            postaction=decorate] (0,0) -- (8.8, 2.8)
            node[midway, below right, color= gray, font=\small]{$\theta$};
        \end{tikzpicture}
        
      \\ \hline \\1b &
          \begin{tikzpicture}
         \filldraw[black] (0,0) circle (2pt) node[anchor=south]{$A$};
        \filldraw[black] (3.6, 2) circle (2pt) node[anchor=south]{$B$};
        \filldraw[black] (5.4, 2) circle (2pt) node[anchor=north]{$C$};
        \filldraw[black] (8.8, 2.8) circle (2pt) node[anchor=north]{$D$};
        \path[draw = gray,% 
             decoration={%
            markings,%
            mark=at position 0.5   with {\arrow[scale=1.5]{>}},%
            },%
            postaction=decorate] (0,0) -- (3.6, 2)
            node[midway, above left, color= gray, font=\small]{$\theta+\delta$};
        \path[draw = gray,% 
             decoration={%
            markings,%
            mark=at position 0.5   with {\arrow[scale=1.5]{<}},%
            },%
            postaction=decorate] (0,0) -- (5.4, 2)
            node[midway, below, color= gray, font=\small]{$\delta$};
        \path[draw = gray,% 
             decoration={%
            markings,%
            mark=at position 0.5   with {\arrow[scale=1.5]{>}},%
            },%
            postaction=decorate] (3.6, 2) -- (8.8, 2.8)
            node[midway, above, color= gray, font=\small]{$\theta$};
        \end{tikzpicture}
\\ \hline \\ 1f &
    \begin{tikzpicture}
         \filldraw[black] (0,0) circle (2pt) node[anchor=south]{$A$};
        \filldraw[black] (3.6, 2) circle (2pt) node[anchor=south]{$B$};
        \filldraw[black] (5.4, 2) circle (2pt) node[anchor=north]{$C$};
        \filldraw[black] (8.8, 2.8) circle (2pt) node[anchor=north]{$D$};
        \path[draw = gray,% 
             decoration={%
            markings,%
            mark=at position 0.5   with {\arrow[scale=1.5]{>}},%
            },%
            postaction=decorate] (0,0) -- (3.6, 2)
            node[midway, above, color= gray, font=\small]{$\theta$};
        \path[draw = gray,% 
             decoration={%
            markings,%
            mark=at position 0.5   with {\arrow[scale=1.5]{>}},%
            },%
            postaction=decorate] (3.6, 2) -- (8.8, 2.8)
            node[midway, above, color= gray, font=\small]{$\theta$};
        \path[draw = gray,% 
             decoration={%
            markings,%
            mark=at position 0.5   with {\arrow[scale=1.5]{<}},%
            },%
            postaction=decorate] (3.6, 2) -- (5.4, 2)
            node[midway, below, color= gray, font=\small]{$\delta$};
    \end{tikzpicture}
\\ \hline
     \\ 1k &
\begin{tikzpicture}
     \filldraw[black] (0,0) circle (2pt) node[anchor=south]{$A$};
    \filldraw[black] (3.6, 2) circle (2pt) node[anchor=south]{$B$};
    \filldraw[black] (5.4, 2) circle (2pt) node[anchor=south]{$C$};
    \filldraw[black] (8.8, 2.8) circle (2pt) node[anchor=north]{$D$};
    \path[draw = gray,% 
             decoration={%
            markings,%
            mark=at position 0.5   with {\arrow[scale=1.5]{<}},%
            },%
            postaction=decorate] (0,0) -- (5.4, 2)
            node[midway, above, color= gray, font=\small]{$\delta$};
    \path[draw = gray,% 
             decoration={%
            markings,%
            mark=at position 0.5   with {\arrow[scale=1.5]{<}},%
            },%
            postaction=decorate] (3.6, 2) -- (8.8, 2.8)
            node[midway, above, color= gray, font=\small]{$\delta$};
    \path[draw = gray,% 
             decoration={%
            markings,%
            mark=at position 0.5   with {\arrow[scale=1.5]{>}},%
            },%
            postaction=decorate] (0,0) -- (8.8, 2.8)
            node[midway, below, color= gray, font=\small]{$\theta+\delta$};
\end{tikzpicture}
\\ \hline
\\ 1l &
\begin{tikzpicture}
     \filldraw[black] (0,0) circle (2pt) node[anchor=south]{$A$};
    \filldraw[black] (3.6, 2) circle (2pt) node[anchor=south]{$B$};
    \filldraw[black] (5.4, 2) circle (2pt) node[anchor=north]{$C$};
    \filldraw[black] (8.8, 2.8) circle (2pt) node[anchor=north]{$D$};
    \path[draw = gray,% 
             decoration={%
            markings,%
            mark=at position 0.5   with {\arrow[scale=1.5]{>}},%
            },%
            postaction=decorate] (0,0) -- (5.4, 2)
            node[midway, above, color= gray, font=\small]{$\theta$};
    \path[draw = gray,% 
             decoration={%
            markings,%
            mark=at position 0.5   with {\arrow[scale=1.5]{<}},%
            },%
            postaction=decorate] (3.6, 2) -- (5.4, 2)
            node[midway, below, color= gray, font=\small]{$\theta+\delta$};
    \path[draw = gray,% 
             decoration={%
            markings,%
            mark=at position 0.5   with {\arrow[scale=1.5]{>}},%
            },%
            postaction=decorate] (3.6, 2) -- (8.8, 2.8)
            node[midway, above, color= gray, font=\small]{$\theta$};
\end{tikzpicture}
\\ \hline \\ 1m &
\begin{tikzpicture}
     \filldraw[black] (0,0) circle (2pt) node[anchor=south]{$A$};
    \filldraw[black] (3.6, 2) circle (2pt) node[anchor=south]{$B$};
    \filldraw[black] (5.4, 2) circle (2pt) node[anchor=south]{$C$};
    \filldraw[black] (8.8, 2.8) circle (2pt) node[anchor=north]{$D$};
    \path[draw = gray,% 
             decoration={%
            markings,%
            mark=at position 0.5   with {\arrow[scale=1.5]{>}},%
            },%
            postaction=decorate] (0,0) -- (5.4, 2)
            node[midway, above, color= gray, font=\small]{$\theta$};
    \path[draw = gray,% 
             decoration={%
            markings,%
            mark=at position 0.5   with {\arrow[scale=1.5]{<}},%
            },%
            postaction=decorate] (3.6, 2) -- (5.4, 2)
            node[midway, above, color= gray, font=\small]{$\delta$};
    \path[draw = gray,% 
             decoration={%
            markings,%
            mark=at position 0.5   with {\arrow[scale=1.5]{>}},%
            },%
            postaction=decorate] (5.4, 2) -- (8.8, 2.8)
            node[midway, above, color= gray, font=\small]{$\theta$};
\end{tikzpicture}
\\ \hline
\\ 1o &
\begin{tikzpicture}
     \filldraw[black] (0,0) circle (2pt) node[anchor=south]{$A$};
    \filldraw[black] (3.6, 2) circle (2pt) node[anchor=south]{$B$};
    \filldraw[black] (5.4, 2) circle (2pt) node[anchor=north]{$C$};
    \filldraw[black] (8.8, 2.8) circle (2pt) node[anchor=north]{$D$};
       \path[draw = gray,% 
             decoration={%
            markings,%
            mark=at position 0.5   with {\arrow[scale=1.5]{>}},%
            },%
            postaction=decorate] (0,0) -- (5.4, 2)
            node[midway, above, color= gray, font=\small]{$\theta$};
       \path[draw = gray,% 
             decoration={%
            markings,%
            mark=at position 0.5   with {\arrow[scale=1.5]{>}},%
            },%
            postaction=decorate] (5.4, 2) -- (8.8, 2.8)
            node[midway, below, color= gray, font=\small]{$\theta+\delta$};
       \path[draw = gray,% 
             decoration={%
            markings,%
            mark=at position 0.5   with {\arrow[scale=1.5]{<}},%
            },%
            postaction=decorate] (3.6, 2) -- (8.8, 2.8)
            node[midway, above, color= gray, font=\small]{$\delta$};
\end{tikzpicture}\\ \hline  \\ 1s &
\begin{tikzpicture}
     \filldraw[black] (0,0) circle (2pt) node[anchor=south]{$A$};
    \filldraw[black] (3.6, 2) circle (2pt) node[anchor=south]{$B$};
    \filldraw[black] (5.4, 2) circle (2pt) node[anchor=south]{$C$};
    \filldraw[black] (8.8, 2.8) circle (2pt) node[anchor=north]{$D$};
    \path[draw = gray,% 
             decoration={%
            markings,%
            mark=at position 0.5   with {\arrow[scale=1.5]{>}},%
            },%
            postaction=decorate] (0,0) -- (8.8, 2.8)
            node[midway, below, color= gray, font=\small]{$\theta$};
    \path[draw = gray,% 
             decoration={%
            markings,%
            mark=at position 0.5   with {\arrow[scale=1.5]{<}},%
            },%
            postaction=decorate] (3.6, 2) -- (8.8, 2.8)
            node[midway, above, color= gray, font=\small]{$\delta$};
    \path[draw = gray,% 
             decoration={%
            markings,%
            mark=at position 0.5   with {\arrow[scale=1.5]{>}},%
            },%
            postaction=decorate] (5.4, 2) -- (8.8, 2.8)
            node[midway, below, color= gray, font=\small]{$\delta$};
\end{tikzpicture}\\ \hline
\caption{Representation of (1a), (1b), (1f), (1k), (1l), (1m), (1o), (1s).}
\end{longtable}

For $j$ sufficiently large and for $k\geq k_0(\alpha)$, the $\alpha$-mass corresponding to (1c) is (see Table \ref{Figch})
\begin{equation}\label{e:choicek1}
\begin{split}
&\theta^\alpha(\Haus^1(\sigma_{C_jD_j})+(1-k^{-1})^{\alpha}\Haus^1(\sigma_{A_jC_j})+k^{-\alpha}\Haus^1(\sigma_{A_jB_j}))\\
&> \theta^\alpha(\Haus^1(\sigma_{C_jD_j})+((1-k^{-1})^{\alpha}+k^{-\alpha})\Haus^1(\sigma_{A_jB_j}))\\
&> \theta^\alpha(\Haus^1(\sigma_{C_jD_j})+\Haus^1(\sigma_{A_jB_j})+\frac{k^{-\alpha}}{2}\Haus^1(\sigma_{A_jB_j}))\\
&> \theta^\alpha(\Haus^1(\sigma_{C_jD_j})+\Haus^1(\sigma_{A_jB_j})+k^{-\alpha}\Haus^1(\sigma_{B_jC_j}))=\MM(W_j).
\end{split}
\end{equation}
Also, for $j$ sufficiently large and for $k\geq k_0(\alpha)$, the $\alpha$-mass corresponding to (1h) is (see Table \ref{Figch})
\begin{equation}\label{e:choicek2}
\begin{split}
&\theta^\alpha(\Haus^1(\sigma_{A_jB_j})+(1-k^{-1})^{\alpha}\Haus^1(\sigma_{B_jD_j})+k^{-\alpha}\Haus^1(\sigma_{C_jD_j}))\\
&> \theta^\alpha(\Haus^1(\sigma_{A_jB_j})+((1-k^{-1})^{\alpha}+k^{-\alpha})\Haus^1(\sigma_{C_jD_j}))\\
&> \theta^\alpha(\Haus^1(\sigma_{A_jB_j})+\Haus^1(\sigma_{C_jD_j})+\frac{k^{-\alpha}}{2}\Haus^1(\sigma_{C_jD_j}))\\
&> \theta^\alpha(\Haus^1(\sigma_{A_jB_j})+\Haus^1(\sigma_{C_jD_j})+k^{-\alpha}\Haus^1(\sigma_{B_jC_j}))=\MM(W_j).
\end{split}
\end{equation}

\begin{longtable}{  p{1cm} c } \label{Figch} 
        \\1c &
        \begin{tikzpicture}
     \filldraw[black] (0,0) circle (2pt) node[anchor=south]{$A$};
    \filldraw[black] (3.6, 2) circle (2pt) node[anchor=south]{$B$};
    \filldraw[black] (5.4, 2) circle (2pt) node[anchor=south]{$C$};
    \filldraw[black] (8.8, 2.8) circle (2pt) node[anchor=north]{$D$};
    \path[draw = gray,% 
             decoration={%
            markings,%
            mark=at position 0.5   with {\arrow[scale=1.5]{>}},%
            },%
            postaction=decorate] (0,0) -- (3.6, 2)
            node[midway, above, color= gray, font=\small]{$\delta$};
    \path[draw = gray,% 
             decoration={%
            markings,%
            mark=at position 0.5   with {\arrow[scale=1.5]{>}},%
            },%
            postaction=decorate] (0,0) -- (5.4, 2)
            node[midway, below, color= gray, font=\small]{$\theta-\delta$};
    \path[draw = gray,% 
             decoration={%
            markings,%
            mark=at position 0.5   with {\arrow[scale=1.5]{>}},%
            },%
            postaction=decorate] (5.4, 2) -- (8.8, 2.8)
            node[midway, below, color= gray, font=\small]{$\theta$};
\end{tikzpicture}
\\ \hline
\\ 1h &
\begin{tikzpicture}
     \filldraw[black] (0,0) circle (2pt) node[anchor=south]{$A$};
    \filldraw[black] (3.6, 2) circle (2pt) node[anchor=north]{$B$};
    \filldraw[black] (5.4, 2) circle (2pt) node[anchor=north]{$C$};
    \filldraw[black] (8.8, 2.8) circle (2pt) node[anchor=north]{$D$};
    \path[draw = gray,% 
             decoration={%
            markings,%
            mark=at position 0.5   with {\arrow[scale=1.5]{>}},%
            },%
            postaction=decorate] (0,0) -- (3.6, 2)
            node[midway, above, color= gray, font=\small]{$\theta$};
    \path[draw = gray,% 
             decoration={%
            markings,%
            mark=at position 0.5   with {\arrow[scale=1.5]{>}},%
            },%
            postaction=decorate] (5.4, 2) -- (8.8, 2.8)
            node[midway, below, color= gray, font=\small]{$\delta$};
    \path[draw = gray,% 
             decoration={%
            markings,%
            mark=at position 0.5   with {\arrow[scale=1.5]{>}},%
            },%
            postaction=decorate] (3.6, 2) -- (8.8, 2.8)
            node[midway, above, color= gray, font=\small]{$\theta-\delta$};
\end{tikzpicture} \\ \hline 
  \caption{Representation of (1c), (1h).}
\end{longtable}

Lastly, we exclude case (1e) by direct comparison with the $\alpha$-mass of $Z_j$. For $j$ sufficiently large and for $k\geq k_0(\alpha)$, the $\alpha$-mass corresponding to (1e) is (see Table \ref{Fige})
\begin{equation}\label{e:choicek3}
\begin{split}
&\theta^\alpha((1-k^{-1})^{\alpha}\Haus^1(\sigma_{A_jD_j})+k^{-\alpha}(\Haus^1(\sigma_{A_jB_j})+\Haus^1(\sigma_{C_jD_j})))\\
&> \theta^\alpha((1-k^{-1})^{\alpha}\Haus^1(\sigma_{A_jD_j})+k^{-\alpha}\frac{1}{2}\Haus^1(\sigma_{A_jD_j}))\\
&> \theta^\alpha(\Haus^1(\sigma_{A_jD_j})+\frac{1}{4}k^{-\alpha}\Haus^1(\sigma_{A_jD_j}))\\
&> \theta^\alpha(\Haus^1(\sigma_{A_jD_j})+k^{-\alpha}\Haus^1(\sigma_{B_jC_j}))=\MM(Z_j).
\end{split}
\end{equation}

\begin{longtable}{  p{1cm} c } \label{Fige}
   \\ 1e &
\begin{tikzpicture}
     \filldraw[black] (0,0) circle (2pt) node[anchor=south]{$A$};
    \filldraw[black] (3.6, 2) circle (2pt) node[anchor=south]{$B$};
    \filldraw[black] (5.4, 2) circle (2pt) node[anchor=south]{$C$};
    \filldraw[black] (8.8, 2.8) circle (2pt) node[anchor=north]{$D$};
    \path[draw = gray,% 
             decoration={%
            markings,%
            mark=at position 0.5   with {\arrow[scale=1.5]{>}},%
            },%
            postaction=decorate] (0,0) -- (3.6, 2)
            node[midway, above, color= gray, font=\small]{$\delta$};
    \path[draw = gray,% 
             decoration={%
            markings,%
            mark=at position 0.5   with {\arrow[scale=1.5]{>}},%
            },%
            postaction=decorate] (0,0) -- (8.8, 2.8)
            node[midway, below, color= gray, font=\small]{$\theta-\delta$};
    \path[draw = gray,% 
             decoration={%
            markings,%
            mark=at position 0.5   with {\arrow[scale=1.5]{>}},%
            },%
            postaction=decorate] (5.4, 2) -- (8.8, 2.8)
            node[midway, above, color= gray, font=\small]{$\delta$};
    \end{tikzpicture}
\\ \hline
\caption{Representation of (1e).}
\end{longtable}

\emph{Sub-case 1-3}. The last situation which we need to take into account is when exactly three points are collinear. We will discuss the case in which the collinear points are $A_j$, $B_j$ and $C_j$ or $A_j, C_j$ and $D_j$. The remaining cases in which the collinear points are $B_j$, $C_j$ and $D_j$ or $A_j, B_j$ and $D_j$ are symmetric and can be treated analogously, therefore we leave the analysis to the reader.\\

\emph{Sub-case 1-3-1: $A_j$, $B_j$ and $C_j$ are collinear}.

The cases (1d), (1i), (1j), (1q), (1r) can be excluded for the same reason as in the Sub-case 1-2 (see Table \ref{Fig_wrong}).

We exclude cases (1b), (1f), (1l), (1n), which are coincident (see Table \ref{Fignall}), by direct comparison with the $\alpha$-mass of $Z_j$. For $j$ sufficiently large, the $\alpha$-mass corresponding to the above cases is
    $$\theta^\alpha(\Haus^1(\sigma_{A_jB_j})+\Haus^1(\sigma_{B_jD_j})+k^{-\alpha}\Haus^1(\sigma_{B_jC_j}))>\MM(Z_j).$$

\begin{longtable}{  p{1cm} c } \label{Fig_wrong} 
\\ 1d &
\begin{tikzpicture}
     \filldraw[black] (0,0) circle (2pt) node[anchor=south]{$A$};
    \filldraw[black] (1.8, 0) circle (2pt) node[anchor=south]{$B$};
    \filldraw[black] (2.6, 0) circle (2pt) node[anchor=south]{$C$};
    \filldraw[black] (4.4, 1.4) circle (2pt) node[anchor=north]{$D$};
    \draw[gray, thick] (0,0) -- (1.8, 0);
    \draw[gray, thick] (0,0) -- (4.4, 1.4);
    \draw[gray, thick] (1.8, 0) -- (2.6, 0);
\end{tikzpicture}
  \qquad \qquad
\begin{tikzpicture}
     \filldraw[black] (0,0) circle (2pt) node[anchor=south]{$A$};
    \filldraw[black] (1.8, 0) circle (2pt) node[anchor=south]{$B$};
    \filldraw[black] (2.6, 0) circle (2pt) node[anchor=south]{$C$};
    \filldraw[black] (4.4, 1.4) circle (2pt) node[anchor=north]{$D$};
    \draw[gray, thick] (0,0) -- (1.8, 0);
    \draw[gray, thick] (2.6, 0) -- (4.4, 1.4);
\end{tikzpicture}
\quad  1i 
\\ \hline 
\\ 1j &
\begin{tikzpicture}
     \filldraw[black] (0,0) circle (2pt) node[anchor=south]{$A$};
    \filldraw[black] (1.8, 0) circle (2pt) node[anchor=south]{$B$};
    \filldraw[black] (2.6, 0) circle (2pt) node[anchor=south]{$C$};
    \filldraw[black] (4.4, 1.4) circle (2pt) node[anchor=north]{$D$};
    \draw[gray, thick] (0,0) -- (2.6, 0);
    \draw[gray, thick] (1.8, 0) -- (2.6, 0);
    \draw[gray, thick] (0,0) -- (4.4, 1.4);
\end{tikzpicture}
 \qquad \qquad
\begin{tikzpicture}
\filldraw[black] (0,0) circle (2pt) node[anchor=south]{$A$};
    \filldraw[black] (1.8, 0) circle (2pt) node[anchor=south]{$B$};
    \filldraw[black] (2.6, 0) circle (2pt) node[anchor=south]{$C$};
    \filldraw[black] (4.4, 1.4) circle (2pt) node[anchor=north]{$D$};
    \draw[gray, thick] (0,0) -- (4.4, 1.4);
    \draw[gray, thick] (1.8, 0) -- (2.6, 0);
    \draw[gray, thick] (1.8, 0) -- (4.4, 1.4);
\end{tikzpicture}
\quad 1q
\\ \hline 
\\ 1r &
\begin{tikzpicture}
   \filldraw[black] (0,0) circle (2pt) node[anchor=south]{$A$};
    \filldraw[black] (1.8, 0) circle (2pt) node[anchor=south]{$B$};
    \filldraw[black] (2.6, 0) circle (2pt) node[anchor=south]{$C$};
    \filldraw[black] (4.4, 1.4) circle (2pt) node[anchor=north]{$D$};
    \draw[gray, thick] (0,0) -- (4.4, 1.4);
    \draw[gray, thick] (1.8, 0) -- (2.6, 0);
    \draw[gray, thick] (2.6, 0) -- (4.4, 1.4);  
\end{tikzpicture}
\\ \hline
\caption{Representation of (1d), (1i), (1j), (1q), (1r) in the collinear case.} 
\end{longtable} 

\begin{longtable}{  p{1cm} c } \label{Fignall}
   \\ &
\begin{tikzpicture}
     \filldraw[black] (0,0) circle (2pt) node[anchor=south]{$A$};
    \filldraw[black] (3.6, 0) circle (2pt) node[anchor=south]{$B$};
    \filldraw[black] (5.4, 0) circle (2pt) node[anchor=south]{$C$};
    \filldraw[black] (8.8, 2.8) circle (2pt) node[anchor=north]{$D$};
    \path[draw = gray,% 
             decoration={%
            markings,%
            mark=at position 0.5   with {\arrow[scale=1.5]{>}},%
            },%
            postaction=decorate] (0,0) -- (3.6, 0)
            node[midway, above, color= gray, font=\small]{$\theta$};
    \path[draw = gray,% 
             decoration={%
            markings,%
            mark=at position 0.5   with {\arrow[scale=1.5]{>}},%
            },%
            postaction=decorate] (3.6,0) -- (8.8, 2.8)
            node[midway, above, color= gray, font=\small]{$\theta$};
    \path[draw = gray,% 
             decoration={%
            markings,%
            mark=at position 0.5   with {\arrow[scale=1.5]{>}},%
            },%
            postaction=decorate] (5.4, 0) -- (3.6, 0)
            node[midway, above, color= gray, font=\small]{$\delta$};
    \end{tikzpicture}
\\ \hline
\caption{Representation of (1b), (1f), (1l), (1n) in the collinear case.}
\end{longtable}

We exclude case (1a), since it coincides with case (1j), which we have already excluded and we exclude cases (1k) and (1o) because they contain a loop (see Table \ref{Figloop}).
\begin{longtable}{  p{1cm} c } \label{Figloop} 
\\ 1k &
\begin{tikzpicture}
     \filldraw[black] (0,0) circle (2pt) node[anchor=south]{$A$};
    \filldraw[black] (1.8, 0) circle (2pt) node[anchor=south]{$B$};
    \filldraw[black] (2.6, 0) circle (2pt) node[anchor=south]{$C$};
    \filldraw[black] (4.4, 1.4) circle (2pt) node[anchor=north]{$D$};
    \draw[gray, thick] (0,0) -- (2.6, 0);
    \draw[gray, thick] (0,0) -- (4.4, 1.4);
    \draw[gray, thick] (1.8, 0) -- (4.4, 1.4);
\end{tikzpicture}
  \qquad \qquad
\begin{tikzpicture}
     \filldraw[black] (0,0) circle (2pt) node[anchor=south]{$A$};
    \filldraw[black] (1.8, 0) circle (2pt) node[anchor=south]{$B$};
    \filldraw[black] (2.6, 0) circle (2pt) node[anchor=south]{$C$};
    \filldraw[black] (4.4, 1.4) circle (2pt) node[anchor=north]{$D$};
    \draw[gray, thick] (0,0) -- (2.6, 0);
    \draw[gray, thick] (2.6, 0) -- (4.4, 1.4);
    \draw[gray, thick] (1.8, 0) -- (4.4, 1.4);
\end{tikzpicture}
\quad  1o 
\\ \hline 
\caption{Representation of (1k) and (1o) in the collinear case.}
\end{longtable}

We do not need to exclude cases (1c) and (1m), since the current coincides with $Z_j$ (see Table \ref{Figcm}).

\begin{longtable}{  p{1cm} c } \label{Figcm}
   \\ &
\begin{tikzpicture}
     \filldraw[black] (0,0) circle (2pt) node[anchor=south]{$A$};
    \filldraw[black] (3.6, 0) circle (2pt) node[anchor=south]{$B$};
    \filldraw[black] (5.4, 0) circle (2pt) node[anchor=south]{$C$};
    \filldraw[black] (8.8, 2.8) circle (2pt) node[anchor=north]{$D$};
    \path[draw = gray,% 
             decoration={%
            markings,%
            mark=at position 0.5   with {\arrow[scale=1.5]{>}},%
            },%
            postaction=decorate] (0,0) -- (3.6, 0)
            node[midway, above, color= gray, font=\small]{$\theta$};
    \path[draw = gray,% 
             decoration={%
            markings,%
            mark=at position 0.5   with {\arrow[scale=1.5]{>}},%
            },%
            postaction=decorate] (5.4,0) -- (8.8, 2.8)
            node[midway, above, color= gray, font=\small]{$\theta$};
    \path[draw = gray,% 
             decoration={%
            markings,%
            mark=at position 0.5   with {\arrow[scale=1.5]{>}},%
            },%
            postaction=decorate] (3.6, 0) -- (5.4, 0)
            node[midway, above, color= gray, font=\small]{$\theta-\delta$};
    \end{tikzpicture}
\\ \hline
\caption{Representation of (1c), (1m) in the collinear case.} \endlastfoot
\end{longtable}

Lastly, cases (1e), (1h), (1s) can be excluded with the same argument used in Sub-case 1-2, since the segments in the corresponding support are in general position also when $A_j, B_j$, and $C_j$ are collinear (see Table \ref{Fehs}).

\begin{longtable}{  p{1cm} c } \label{Fehs} 
\\ 1e &
\begin{tikzpicture}
     \filldraw[black] (0,0) circle (2pt) node[anchor=south]{$A$};
    \filldraw[black] (1.8, 0) circle (2pt) node[anchor=south]{$B$};
    \filldraw[black] (2.6, 0) circle (2pt) node[anchor=south]{$C$};
    \filldraw[black] (4.4, 1.4) circle (2pt) node[anchor=north]{$D$};
    \draw[gray, thick] (0,0) -- (1.8, 0);
    \draw[gray, thick] (4.4, 1.4) -- (2.6, 0);
    \draw[gray, thick] (0,0) -- (4.4, 1.4);
\end{tikzpicture}
 \qquad \qquad
\begin{tikzpicture}
\filldraw[black] (0,0) circle (2pt) node[anchor=south]{$A$};
    \filldraw[black] (1.8, 0) circle (2pt) node[anchor=south]{$B$};
    \filldraw[black] (2.6, 0) circle (2pt) node[anchor=south]{$C$};
    \filldraw[black] (4.4, 1.4) circle (2pt) node[anchor=north]{$D$};
    \draw[gray, thick] (2.6,0) -- (4.4, 1.4);
    \draw[gray, thick] (1.8, 0) -- (0, 0);
    \draw[gray, thick] (1.8, 0) -- (4.4, 1.4);
\end{tikzpicture}
\quad 1h
\\ \hline 
\\ 1s &
\begin{tikzpicture}
   \filldraw[black] (0,0) circle (2pt) node[anchor=south]{$A$};
    \filldraw[black] (1.8, 0) circle (2pt) node[anchor=south]{$B$};
    \filldraw[black] (2.6, 0) circle (2pt) node[anchor=south]{$C$};
    \filldraw[black] (4.4, 1.4) circle (2pt) node[anchor=north]{$D$};
    \draw[gray, thick] (0,0) -- (4.4, 1.4);
    \draw[gray, thick] (4.4, 1.4) -- (1.8, 0);
    \draw[gray, thick] (2.6, 0) -- (4.4, 1.4);  
\end{tikzpicture}
\\ \hline
\caption{Representation of (1e), (1h), (1s) in the collinear case.}
\end{longtable}

\emph{Sub-case 1-3-2: $A_j$, $C_j$ and $D_j$ are collinear}.

The cases (1d), (1i), (1n), (1q) can be excluded for the same reason as in the Sub-case 1-2 (see Table \ref{Fig_wrong2}).

\begin{longtable}{  p{1cm} c } \label{Fig_wrong2} 
\\ 1d &
\begin{tikzpicture}
     \filldraw[black] (0,0) circle (2pt) node[anchor=south]{$A$};
    \filldraw[black] (1.6, 0.6) circle (2pt) node[anchor=south]{$B$};
    \filldraw[black] (2.4, 0.6) circle (2pt) node[anchor=south]{$C$};
    \filldraw[black] (4, 1) circle (2pt) node[anchor=north]{$D$};
    \draw[gray, thick] (0,0) -- (1.6, 0.6);
    \draw[gray, thick] (0,0) -- (4, 1);
    \draw[gray, thick] (1.6, 0.6) -- (2.4, 0.6);
\end{tikzpicture}
  \qquad \qquad
\begin{tikzpicture}
     \filldraw[black] (0,0) circle (2pt) node[anchor=south]{$A$};
    \filldraw[black] (1.6, 0.6) circle (2pt) node[anchor=south]{$B$};
    \filldraw[black] (2.4, 0.6) circle (2pt) node[anchor=south]{$C$};
    \filldraw[black] (4, 1) circle (2pt) node[anchor=north]{$D$};
    \draw[gray, thick] (0,0) -- (1.6, 0.6);
    \draw[gray, thick] (2.4, 0.6) -- (4, 1);
\end{tikzpicture}
\quad  1i 
\\ \hline 
\\ 1n &
\begin{tikzpicture}
     \filldraw[black] (0,0) circle (2pt) node[anchor=south]{$A$};
    \filldraw[black] (1.6, 0.6) circle (2pt) node[anchor=south]{$B$};
    \filldraw[black] (2.4, 0.6) circle (2pt) node[anchor=north]{$C$};
    \filldraw[black] (4, 1) circle (2pt) node[anchor=north]{$D$};
    \draw[gray, thick] (0,0) -- (2.4, 0.6);
    \draw[gray, thick] (1.6, 0.6) -- (4, 1);
\end{tikzpicture}
 \qquad \qquad
\begin{tikzpicture}
\filldraw[black] (0,0) circle (2pt) node[anchor=south]{$A$};
    \filldraw[black] (1.6, 0.6) circle (2pt) node[anchor=south]{$B$};
    \filldraw[black] (2.4, 0.6) circle (2pt) node[anchor=north]{$C$};
    \filldraw[black] (4, 1) circle (2pt) node[anchor=north]{$D$};
    \draw[gray, thick] (0,0) -- (4, 1);
    \draw[gray, thick] (1.6, 0.6) -- (2.4, 0.6);
    \draw[gray, thick] (1.6, 0.6) -- (4, 1);
\end{tikzpicture}
\quad 1q
\\ \hline 
\caption{Representation of (1d), (1i), (1n), (1q) in the collinear case.}
\end{longtable}

We exclude cases (1k), (1o), (1s), which are coincident (see Table \ref{Fignall2first}), by direct comparison with the $\alpha$-mass of $Z_j$. For $j$ sufficiently large, the $\alpha$-mass corresponding to the above cases is
    $$\theta^\alpha((1+k^{-1})^\alpha\Haus^1(\sigma_{C_jD_j})+\Haus^1(\sigma_{A_jC_j})+k^{-\alpha}\Haus^1(\sigma_{B_jD_j}))>\MM(Z_j).$$

\begin{longtable}{  p{1cm} c } \label{Fignall2first}
   \\ &
\begin{tikzpicture}
     \filldraw[black] (0,0) circle (2pt) node[anchor=south]{$A$};
    \filldraw[black] (3.2, 1.2) circle (2pt) node[anchor=south]{$B$};
    \filldraw[black] (4.8, 1.2) circle (2pt) node[anchor=north]{$C$};
    \filldraw[black] (8, 2) circle (2pt) node[anchor=north]{$D$};
    \path[draw = gray,% 
             decoration={%
            markings,%
            mark=at position 0.5   with {\arrow[scale=1.5]{>}},%
            },%
            postaction=decorate] (0,0) -- (4.8, 1.2)
            node[midway, above, color= gray, font=\small]{$\theta$};
    \path[draw = gray,% 
             decoration={%
            markings,%
            mark=at position 0.5   with {\arrow[scale=1.5]{>}},%
            },%
            postaction=decorate] (4.8, 1.2) -- (8, 2)
            node[midway, below, color= gray, font=\small]{$\theta+\delta$};
    \path[draw = gray,% 
             decoration={%
            markings,%
            mark=at position 0.5   with {\arrow[scale=1.5]{>}},%
            },%
            postaction=decorate] (3.2, 1.2) -- (8, 2)
            node[midway, above, color= gray, font=\small]{$\delta$};
    \end{tikzpicture}
\\ \hline
\caption{Representation of (1k), (1o), (1s) in the collinear case.}
\end{longtable}

We do not exclude cases (1j), (1m) and (1r), since the current coincides with $Z_j$ (see Table \ref{Figcm2}).
\begin{longtable}{  p{1cm} c } \label{Figcm2}
   \\ &
\begin{tikzpicture}
     \filldraw[black] (0,0) circle (2pt) node[anchor=south]{$A$};
    \filldraw[black] (3.2, 1.2) circle (2pt) node[anchor=south]{$B$};
    \filldraw[black] (4.8, 1.2) circle (2pt) node[anchor=south]{$C$};
    \filldraw[black] (8, 2) circle (2pt) node[anchor=north]{$D$};
    \path[draw = gray,% 
             decoration={%
            markings,%
            mark=at position 0.5   with {\arrow[scale=1.5]{>}},%
            },%
            postaction=decorate] (0,0) -- (4.8, 1.2)
            node[midway, above, color= gray, font=\small]{$\theta$};
    \path[draw = gray,% 
             decoration={%
            markings,%
            mark=at position 0.5   with {\arrow[scale=1.5]{>}},%
            },%
            postaction=decorate] (4.8,1.2) -- (8, 2)
            node[midway, above, color= gray, font=\small]{$\theta$};
    \path[draw = gray,% 
             decoration={%
            markings,%
            mark=at position 0.5   with {\arrow[scale=1.5]{>}},%
            },%
            postaction=decorate] (4.8, 1.2) -- (3.2, 1.2)
            node[midway, above, color= gray, font=\small]{$\delta$};
    \end{tikzpicture}
\\ \hline
\caption{Representation of (1j), (1m), (1r) in the collinear case.} \endlastfoot
\end{longtable}

We exclude cases (1a), (1c), (1e), which are coincident (see Table \ref{Fignall2}), by direct comparison with the $\alpha$-mass of $W_j$. For $j$ sufficiently large and for $k\geq k_0(\alpha)$, the $\alpha$-mass corresponding to the above cases is
\begin{equation}\label{e:choicek4}
\begin{split}
&\theta^\alpha(\Haus^1(\sigma_{C_jD_j})+(1-k^{-1})^{\alpha}\Haus^1(\sigma_{A_jC_j})+k^{-\alpha}\Haus^1(\sigma_{A_jB_j}))\\
&> \theta^\alpha(\Haus^1(\sigma_{C_jD_j})+((1-k^{-1})^{\alpha}+k^{-\alpha})\Haus^1(\sigma_{A_jB_j}))\\
&> \theta^\alpha(\Haus^1(\sigma_{C_jD_j})+(1+\frac{1}{2}k^{-\alpha})\Haus^1(\sigma_{A_jB_j}))\\
&> \theta^\alpha(\Haus^1(\sigma_{C_jD_j})+\Haus^1(\sigma_{A_jB_j}) +k^{-\alpha}\Haus^1(\sigma_{B_jC_j}))=\MM(W_j).
\end{split}
\end{equation}

\begin{longtable}{  p{1cm} c } \label{Fignall2}
      \\ &
\begin{tikzpicture}
     \filldraw[black] (0,0) circle (2pt) node[anchor=south]{$A$};
    \filldraw[black] (3.2, 1.2) circle (2pt) node[anchor=south]{$B$};
    \filldraw[black] (4.8, 1.2) circle (2pt) node[anchor=south]{$C$};
    \filldraw[black] (8, 2) circle (2pt) node[anchor=north]{$D$};
    \path[draw = gray,% 
             decoration={%
            markings,%
            mark=at position 0.5   with {\arrow[scale=1.5]{>}},%
            },%
            postaction=decorate] (0,0) -- (4.8, 1.2)
            node[midway, below, color= gray, font=\small]{$\theta-\delta$};
    \path[draw = gray,% 
             decoration={%
            markings,%
            mark=at position 0.5   with {\arrow[scale=1.5]{>}},%
            },%
            postaction=decorate] (4.8,1.2) -- (8, 2)
            node[midway, above, color= gray, font=\small]{$\theta$};
    \path[draw = gray,% 
             decoration={%
            markings,%
            mark=at position 0.5   with {\arrow[scale=1.5]{>}},%
            },%
            postaction=decorate] (0, 0) -- (3.2, 1.2)
            node[midway, above, color= gray, font=\small]{$\delta$};
    \end{tikzpicture}
\\ \hline
\caption{Representation of (1a), (1c), (1e) in the collinear case.} \endlastfoot
\end{longtable}

Lastly, cases (1b), (1f), (1h), (1l) can be excluded with the same argument used in Sub-case 1-2, since the segments in the corresponding support are in general position also when $A_j, C_j$, and $D_j$ are collinear (see Table \ref{Fehs2}).

\begin{longtable}{  p{1cm} c } \label{Fehs2} 
\\ 1b &
\begin{tikzpicture}
     \filldraw[black] (0,0) circle (2pt) node[anchor=south]{$A$};
    \filldraw[black] (1.6, 0.6) circle (2pt) node[anchor=south]{$B$};
    \filldraw[black] (2.4, 0.6) circle (2pt) node[anchor=north]{$C$};
    \filldraw[black] (4, 1) circle (2pt) node[anchor=north]{$D$};
    \draw[gray, thick] (0,0) -- (1.6, 0.6);
    \draw[gray, thick] (4, 1) -- (1.6, 0.6);
    \draw[gray, thick] (0,0) -- (2.4, 0.6);
\end{tikzpicture}
 \qquad \qquad
\begin{tikzpicture}
\filldraw[black] (0,0) circle (2pt) node[anchor=south]{$A$};
    \filldraw[black] (1.6, 0.6) circle (2pt) node[anchor=south]{$B$};
    \filldraw[black] (2.4, 0.6) circle (2pt) node[anchor=north]{$C$};
    \filldraw[black] (4, 1) circle (2pt) node[anchor=north]{$D$};
    \draw[gray, thick] (2.4,0.6) -- (1.6, 0.6);
    \draw[gray, thick] (1.6, 0.6) -- (0, 0);
    \draw[gray, thick] (1.6, 0.6) -- (4, 1);
\end{tikzpicture}
\quad 1f
\\ \hline 
\\ 1h &
\begin{tikzpicture}
     \filldraw[black] (0,0) circle (2pt) node[anchor=south]{$A$};
    \filldraw[black] (1.6, 0.6) circle (2pt) node[anchor=south]{$B$};
    \filldraw[black] (2.4, 0.6) circle (2pt) node[anchor=north]{$C$};
    \filldraw[black] (4, 1) circle (2pt) node[anchor=north]{$D$};
    \draw[gray, thick] (0,0) -- (1.6, 0.6);
    \draw[gray, thick] (4, 1) -- (1.6, 0.6);
    \draw[gray, thick] (2.4,0.6) -- (4, 1);
\end{tikzpicture}
 \qquad \qquad
\begin{tikzpicture}
\filldraw[black] (0,0) circle (2pt) node[anchor=south]{$A$};
    \filldraw[black] (1.6, 0.6) circle (2pt) node[anchor=south]{$B$};
    \filldraw[black] (2.4, 0.6) circle (2pt) node[anchor=north]{$C$};
    \filldraw[black] (4, 1) circle (2pt) node[anchor=north]{$D$};
    \draw[gray, thick] (2.4,0.6) -- (0, 0);
    \draw[gray, thick] (1.6, 0.6) -- (2.4, 0.6);
    \draw[gray, thick] (1.6, 0.6) -- (4, 1);
\end{tikzpicture}
\quad 1l
\\ \hline
\caption{Representation of (1b), (1f), (1h), (1l) in the collinear case.}
\end{longtable}

\emph{Case 2: $\BR(S_{n_j}\trace Q') = \{E_j\}$.} Recalling that $E_j$ is the endpoint of at least three segments in the support of $S_{n_j}\trace Q'$, see the proof of Lemma \ref{l:numberBranch}, the only possibilities are that
$\supp(S_{n_j}\trace Q')$ is one of the following sets (see Table \ref{Fig16}): 
\begin{itemize}
    \item [(2a)] $\sigma_{A_jE_j}\cup\sigma_{B_jE_j}\cup\sigma_{C_jE_j}\cup\sigma$, with $\sigma\neq\sigma_{D_jE_j}$,
    \item [(2b)] $\sigma_{A_jE_j}\cup\sigma_{B_jE_j}\cup\sigma_{D_jE_j}\cup\sigma$, with $\sigma\neq\sigma_{C_jE_j}$,
    \item [(2c)] $\sigma_{A_jE_j}\cup\sigma_{C_jE_j}\cup\sigma_{D_jE_j}\cup\sigma$, with $\sigma\neq\sigma_{B_jE_j}$,
    \item [(2d)] $\sigma_{B_jE_j}\cup\sigma_{C_jE_j}\cup\sigma_{D_jE_j}\cup\sigma$, with $\sigma\neq\sigma_{A_jE_j}$,
    \item [(2e)] $\sigma_{A_jE_j}\cup\sigma_{B_jE_j}\cup\sigma_{C_jE_j}\cup\sigma_{D_jE_j}$.
\end{itemize}  

\afterpage{
 \begin{longtable}{p{2cm} c  }\label{Fig16}
     2a &
        \begin{tikzpicture}
         \filldraw[black] (0,0) circle (2pt) node[anchor=south]{$A$};
        \filldraw[black] (3.6,2) circle (2pt) node[anchor=south]{$B$};
        \filldraw[black] (5.4, 2) circle (2pt) node[anchor=south]{$C$};
        \filldraw[black] (8.8, 1.2) circle (2pt) node[anchor=north]{$D$};
        \filldraw[black] (2.4, 1.1) circle (2pt) node[anchor=north]{$E$};
        \draw[gray, thick] (0,0) -- (2.4, 1.1);
        \draw[gray, thick] (3.6, 2) -- (2.4, 1.1);
        \draw[gray, thick] (5.4, 2) -- (2.4, 1.1);
        \end{tikzpicture}\\ \hline
        2b &
        \begin{tikzpicture}
         \filldraw[black] (0,0) circle (2pt) node[anchor=south]{$A$};
        \filldraw[black] (3.6,2) circle (2pt) node[anchor=south]{$B$};
        \filldraw[black] (5.4, 2) circle (2pt) node[anchor=south]{$C$};
        \filldraw[black] (8.8, 1.2) circle (2pt) node[anchor=north]{$D$};
        \filldraw[black] (3.5, 0.9) circle (2pt) node[anchor=north]{$E$};
        \draw[gray, thick] (0,0) -- (3.5, 0.9);
        \draw[gray, thick] (3.6, 2) -- (3.5, 0.9);
        \draw[gray, thick] (8.8, 1.2) -- (3.5, 0.9);
        \end{tikzpicture}\\ \hline
        2c &
        \begin{tikzpicture}
         \filldraw[black] (0,0) circle (2pt) node[anchor=south]{$A$};
        \filldraw[black] (3.6,2) circle (2pt) node[anchor=south]{$B$};
        \filldraw[black] (5.4, 2) circle (2pt) node[anchor=south]{$C$};
        \filldraw[black] (8.8, 1.2) circle (2pt) node[anchor=north]{$D$};
        \filldraw[black] (5.6, 1.2) circle (2pt) node[anchor=north]{$E$};
        \draw[gray, thick] (0,0) -- (5.6, 1.2);
        \draw[gray, thick] (5.4, 2) -- (5.6, 1.2);
        \draw[gray, thick] (8.8, 1.2) -- (5.6, 1.2);
        \end{tikzpicture}\\ \hline
            2d &
        \begin{tikzpicture}
         \filldraw[black] (0,0) circle (2pt) node[anchor=south]{$A$};
        \filldraw[black] (3.6,2) circle (2pt) node[anchor=south]{$B$};
        \filldraw[black] (5.4, 2) circle (2pt) node[anchor=south]{$C$};
        \filldraw[black] (8.8, 1.2) circle (2pt) node[anchor=north]{$D$};
        \filldraw[black] (6.5, 1.6) circle (2pt) node[anchor=north]{$E$};
        \draw[gray, thick] (3.6,2) -- (6.5, 1.6);
        \draw[gray, thick] (5.4, 2) -- (6.5, 1.6);
        \draw[gray, thick] (8.8, 1.2) -- (6.5, 1.6);
        \end{tikzpicture}\\ \hline
            2e &
        \begin{tikzpicture}
         \filldraw[black] (0,0) circle (2pt) node[anchor=south]{$A$};
        \filldraw[black] (3.6,2) circle (2pt) node[anchor=south]{$B$};
        \filldraw[black] (5.4, 2) circle (2pt) node[anchor=south]{$C$};
        \filldraw[black] (8.8, 1.2) circle (2pt) node[anchor=north]{$D$};
        \filldraw[black] (3.7, 1.1) circle (2pt) node[anchor=north]{$E$};
        \draw[gray, thick] (0,0) -- (3.7, 1.1);
        \draw[gray, thick] (3.6,2) -- (3.7, 1.1);
        \draw[gray, thick] (5.4, 2) -- (3.7, 1.1);
        \draw[gray, thick] (8.8, 1.2) -- (3.7, 1.1);
        \end{tikzpicture}
\\
\hline
\caption{Representation of (2a), (2b), (2c), (2d), (2e). In (2a), (2b), (2c), (2d) we do not represent the segment $\sigma$.}
\end{longtable}}

We exclude case (2a), indeed by \cite[Lemma 12.1 and Lemma 12.2]{BCM}, $E_j\in{\rm{conv}}(\{A_j, B_j, C_j\})$, hence we have for $\rho$ small and $j$ sufficiently large
\begin{equation}\label{e:angolo1}
    \pi-O(\rho)=\angle{A_jB_jC_j}\leq\angle{A_jE_jC_j}\leq\pi \,.
\end{equation} 
This contradicts \cite[Lemma 12.2]{BCM} for $\rho \leq \rho(k)$, since the modulus of the multiplicity of $\sigma_{A_jE_j},\sigma_{E_jB_j}$ and $\sigma_{E_jC_j}$ belongs to $[k^{-1},\Mass(b_{n_j})]$, which by \eqref{e:massa_bienne} is contained in $[k^{-1},(1+hk^{-1})\Mass(b)]$.\\

Cases (2b), (2c) and (2d) are excluded with a similar argument as in case (2a), where the angle $\angle{A_jE_jC_j}$ in \eqref{e:angolo1} is replaced respectively by $\angle{A_jE_jD_j}$, $\angle{A_jE_jD_j}$ and $\angle{B_jE_jD_j}$.\\

We exclude case (2e) by direct comparison with the $\alpha$-mass of $Z_j$. The $\alpha$-mass corresponding to (2e) is
\begin{equation}\label{e:case2e}
\begin{split}
&\theta^\alpha(\Haus^1(\sigma_{A_jE_j})+\Haus^1(\sigma_{E_jD_j})+k^{-\alpha}(\Haus^1(\sigma_{B_jE_j})+\Haus^1(\sigma_{E_jC_j})))\\
&\geq \theta^\alpha(\Haus^1(\sigma_{A_jD_j})+k^{-\alpha}\Haus^1(\sigma_{B_jC_j}))=\MM(Z_j),
\end{split}
\end{equation}
where the inequality is strict unless $\{E_j\}=\sigma_{A_jD_j}\cap \sigma_{B_jC_j}$, namely unless the current is $Z_j$, which of course we do not need to exclude.\\


\emph{Case 3: $\BR(S_{n_j}\trace Q') = \{E_j,F_j\}$.} Recalling \cite[Proposition 7.4]{BCM}, and the fact that both $E_j$ and $F_j$ are the endpoints of at least three segments in the support of $S_{n_j}\trace Q'$, see the proof of Lemma \ref{l:numberBranch}, up to switching between $E_j$ and $F_j$, the only possibilities are that $\supp(S_{n_j}\trace Q')$ is one of the following sets (see Table \ref{Fig17}):
\begin{itemize}
    \item [(3a)] $\sigma_{E_jF_j}\cup\sigma_{A_jE_j}\cup\sigma_{B_jE_j}\cup\sigma_{C_jF_j}\cup\sigma_{D_jF_j}$,
    \item [(3b)] $\sigma_{E_jF_j}\cup\sigma_{A_jE_j}\cup\sigma_{C_jE_j}\cup\sigma_{B_jF_j}\cup\sigma_{D_jF_j}$,
    \item [(3c)] $\sigma_{E_jF_j}\cup\sigma_{A_jE_j}\cup\sigma_{D_jE_j}\cup\sigma_{B_jF_j}\cup\sigma_{C_jF_j}$.
\end{itemize}

(3a) 
Denote by $\pi_0$ the affine 2-plane passing through $A_j$, $B_j$ and $F_j$ (and therefore containing $E_j$ as well). By \cite[Lemma 12.2]{BCM} the line $\ell$ containing $\sigma_{E_jF_j}$ divides $\pi_0\setminus \ell$ into two open half-planes $\pi_0^-$ and $\pi_0^+$ containing respectively $A_j$ and $B_j$. Let $C'_j$ and $D'_j$ denote the orthogonal projections onto $\pi_0$ of $C_j$ and $D_j$ respectively and observe that $C'_j\in \pi_0^+$. This follows from the fact that by \cite[Lemma 12.2]{BCM} there exists a positive constant $\kappa$ (depending on $k$) such that $\angle A_jE_jF_j\leq \pi-\kappa$ and assuming $C'_j\notin \pi_0^+$ would lead to 
$$\angle A_jB_jC'_j \leq\angle A_jE_jF_j\leq \pi-\kappa,$$
which is a contradiction, for $\rho$ sufficiently small with respect to $k$, since, due to the fact that $\angle A_jB_jC_j\geq \frac\pi 2$, 
$$\angle A_jB_jC'_j\geq\angle A_jB_jC_j\geq \pi-O(\rho).$$
On the other hand, the fact that $C'_j\in\pi_0^+$ implies that $D'_j\in\pi_0^-$, hence
$$\angle A_jE_jD'_j\leq\angle A_jE_jF_j\leq \pi-\kappa$$
which is a contradiction, for $\rho$ sufficiently small with respect to $k$, since, as above, 
$$\angle A_jE_jD'_j\geq\angle A_jE_jD_j\geq \pi-O(\rho).$$

\afterpage{
\begin{longtable}{p{2cm} c  }\label{Fig17}
     3a &
\begin{tikzpicture}
         \filldraw[black] (0,0) circle (2pt) node[anchor=south]{$A$};
        \filldraw[black] (4,1.8) circle (2pt) node[anchor=south]{$B$};
        \filldraw[black] (6, 1.8) circle (2pt) node[anchor=south]{$C$};
        \filldraw[black] (10, 0) circle (2pt) node[anchor=north]{$D$};
        \filldraw[black] (3.5, 1) circle (2pt) node[anchor=north]{$E$};
        \filldraw[black] (6.5, 1.1) circle (2pt) node[anchor=north]{$F$};
        \path[draw = gray,% 
             decoration={%
            markings,%
            mark=at position 0.5   with {\arrow[scale=1.5]{>}},%
            },%
            postaction=decorate] (0,0) -- (3.5, 1)
            node[midway, above, color= gray, font=\small]{$\theta$};
        \path[draw = gray,% 
             decoration={%
            markings,%
            mark=at position 0.5   with {\arrow[scale=1.5]{>}},%
            },%
            postaction=decorate] (6.5, 1.1)  -- (10, 0)
            node[midway, above, color= gray, font=\small]{$\theta$};
        \path[draw = gray,% 
             decoration={%
            markings,%
            mark=at position 0.5   with {\arrow[scale=1.5]{<}},%
            },%
            postaction=decorate] (6.5, 1.1)  -- (6, 1.8)
            node[midway, above right, color= gray, font=\small]{$\delta$};
        \path[draw = gray,% 
             decoration={%
            markings,%
            mark=at position 0.5   with {\arrow[scale=1.5]{>}},%
            },%
            postaction=decorate] (3.5, 1) -- (4, 1.8)
            node[midway, above left, color= gray, font=\small]{$\delta$};
        \path[draw = gray,% 
             decoration={%
            markings,%
            mark=at position 0.5   with {\arrow[scale=1.5]{>}},%
            },%
            postaction=decorate] (3.5, 1) -- (6.5, 1.1)
            node[midway, above, color= gray, font=\small]{$\theta-\delta$};
        \end{tikzpicture}\\ \hline
        3b &
\begin{tikzpicture}
         \filldraw[black] (0,0) circle (2pt) node[anchor=south]{$A$};
        \filldraw[black] (5.5, 2.6) circle (2pt) node[anchor= south east]{$B$};
        \filldraw[black] (4.5, 0.6) circle (2pt) node[anchor=north west]{$C$};
        \filldraw[black] (10, 3.0) circle (2pt) node[anchor=north]{$D$};
        \filldraw[black] (5.7, 2.0) circle (2pt) node[anchor=north west]{$F$};
        \filldraw[black] (4.3, 1.2) circle (2pt) node[anchor=south east]{$E$};
         \path[draw = gray,% 
             decoration={%
            markings,%
            mark=at position 0.5   with {\arrow[scale=1.5]{>}},%
            },%
            postaction=decorate] (0,0) -- (4.3, 1.2)
            node[midway, above , color= gray, font=\small]{$\theta$};
         \path[draw = gray,% 
             decoration={%
            markings,%
            mark=at position 0.5   with {\arrow[scale=1.5]{>}},%
            },%
            postaction=decorate] (5.7, 2.0)  -- (10, 3.0)
            node[midway, below, color= gray, font=\small]{$\theta$};
         \path[draw = gray,% 
             decoration={%
            markings,%
            mark=at position 0.6   with {\arrow[scale=1.5]{>}},%
            },%
            postaction=decorate] (5.7, 2.0)  -- (5.5, 2.6)
            node[midway, right , color= gray, font=\small]{$\delta$};
         \path[draw = gray,% 
             decoration={%
            markings,%
            mark=at position 0.5   with {\arrow[scale=1.5]{<}},%
            },%
            postaction=decorate] (5.7, 2.0) -- (4.3, 1.2)
            node[midway, below right, color= gray, font=\small]{$\theta+\delta$};
         \path[draw = gray,% 
             decoration={%
            markings,%
            mark=at position 0.5   with {\arrow[scale=1.5]{<}},%
            },%
            postaction=decorate] (4.3, 1.2) -- (4.5, 0.6)
            node[midway, left, color= gray, font=\small]{$\delta$};
        \end{tikzpicture} \\ \hline
        3c &
\begin{tikzpicture}
         \filldraw[black] (0,0) circle (2pt) node[anchor=south]{$A$};
        \filldraw[black] (4,1.8) circle (2pt) node[anchor=south]{$B$};
        \filldraw[black] (6, 1.8) circle (2pt) node[anchor=south]{$C$};
        \filldraw[black] (10, 0) circle (2pt) node[anchor=north]{$D$};
        \filldraw[black] (4.8, 0.6) circle (2pt) node[anchor=north]{$E$};
        \filldraw[black] (5.2, 1.3) circle (2pt) node[anchor=south]{$F$};
        \draw[gray, thick] (0,0) -- (4.8, 0.6);
        \draw[gray, thick] (4.8, 0.6)  -- (10, 0);
        \draw[gray, thick] (4.8, 0.6)  -- (5.2, 1.3);
        \draw[gray, thick] (5.2, 1.3) -- (4,1.8);
        \draw[gray, thick] (5.2, 1.3) -- (6, 1.8);
        \end{tikzpicture}\\
\hline
\caption{Representation of (3a), (3b), (3c).} 
\end{longtable}}

(3b) By \cite[Lemma 12.2]{BCM} applied at the branch point $E_j$ we deduce that the angle between the oriented segments $\sigma_{A_j E_j}$ and $\sigma_{E_j F_j}$ tends to 0 as $k\to\infty$. By the same argument applied at the branch point $F_j$ we deduce the same property for the angle between the oriented segments $\sigma_{E_j F_j}$ and $\sigma_{F_j D_j}$. As a consequence, the angle between the oriented segments $\sigma_{A_j D_j}$ and $\sigma_{E_j F_j}$ tends to 0 as $k\to\infty$. Again by \cite[Lemma 12.2]{BCM}, the angles $\angle C_jE_jF_j$ and $\angle E_jF_jB_j$ are equal to $\frac{\pi}{2}+C(k)$ where $C(k)$ tends to 0 as $k\to\infty$. 

Next, using that the angle $\angle E_jF_jD_j$ differs from $\pi$ by a positive constant which depends only on $k$, we observe that the plane containing $A_j,C_j,F_j$ (and therefore also $E_j$) is obtained from the plane containing $D_j,B_j,E_j$ (and therefore also $F_j$) by a rotation $O$ around the line containing $\sigma_{E_jF_j}$ such that, for any fixed $k$,  $\|O-Id\|<f(\rho)$, where $f(\rho)$ tends to 0 as $\rho\to 0$. This implies that the angle between the oriented segments $\sigma_{B_jC_j}$ and $\sigma_{A_jE_j}$ is larger than $\frac{\pi}{2}-c(k)$, where $c(k)$ tends to 0 as $k\to\infty$. This is a contradiction for $\rho$ sufficiently small and $k$ sufficiently large, since for any $k$ the angle between the oriented segments $\sigma_{B_jC_j}$ and $\sigma_{A_jD_j}$ tends to 0 as $\rho\to 0$ and for any $\rho$ the angle between the oriented segment $\sigma_{A_jE_j}$ and the oriented segment $\sigma_{A_jD_j}$ tends to 0 (independently of $\rho$) as $k\to \infty$.

(3c) We exclude this case as the corresponding set is not the support of any
current with boundary $ \partial (S_{n_j} \trace Q') $, because both the segments $\sigma_{A_j E_j} $ and $\sigma_{E_j D_j}$ should have multiplicity $\theta$, thus the multiplicity of $\sigma_{E_j F_j}$ would be zero.

\subsection{Conclusion.}\label{Step3} 
In order to conclude the proof of Lemma \ref{l:main}, for $j$ and $k$ sufficiently large we now exclude the case in which $S_{n_j}$ coincides with $Z_j$ close to at least one point $p_i$. Indeed, should that happen, we could build a better competitor than $T$ for $b$ by adding $\frac1k \sum_{i=1}^h T\trace B_{{n_j}^{-1}}(p_i)$ to $S_{n_j}$. 

We claim that, for $j$ sufficiently large and for $k\geq k_0(\alpha)$, we have
        \begin{equation}\label{e:finalclaim}
            \MM \left(S_{n_j}+\frac1k \sum_{i=1}^h T\trace B_{{n_j}^{-1}}(p_i)\right)\leq \MM(T),
        \end{equation}
the inequality being strict unless $S$ passes through all the $p_i$'s and around every $p_i$ the current $S_{n_j}$ has the shape $W_j$, i.e. (remember that $Q'$ depends on $i$)
\begin{equation}\label{e:finalinclusion}
\{i:p_i\notin\supp(S)\}=\emptyset=\{i:S_{n_j}\trace Q'=Z_j\}.    
\end{equation}
The validity of the claim would conclude the proof, as the following argument shows. The inequality \eqref{e:finalclaim} cannot be strict, since $\partial (S_{n_j}+\frac1k \sum_{i=1}^h T\trace B_{{n_j}^{-1}}(p_i))=b$ by \eqref{e:bienne} and $T\in\OTP(b)$. We deduce that \eqref{e:finalinclusion} holds, which by Lemma \ref{l:magic_points} implies that $S=T$ and $S_{n_j}+\frac1k \sum_{i=1}^h T\trace B_{{n_j}^{-1}}(p_i)=T$ and therefore, recalling \eqref{e:bienne}, we conclude that $S_{n_j}=T_{n_j}$ for $j$ sufficiently large.
\pagebreak

In order to prove the claim, let us consider firstly the case $S\neq T$. Let $p\in\{p_1,\dots,p_h\}\setminus\supp(S)$ and take $r>0$ such that 
\begin{equation}\label{e:p_isolated}
p\notin B_{3r} \big(\supp(S)\cup (\{p_1, \dots, p_h\}\setminus \{p\}) \big).    
\end{equation}
Applying Lemma \ref{l:conv_hauss}, for $j$ sufficiently large we have
\begin{equation}\label{e:sn_dentro}
    \begin{split}
        \supp(S_{n_j}) &\subset B_r(\supp(S)\cup\{p_1,\dots,p_h\})\\
        &=B_r(\supp(S)\cup\{p_1,\dots,p_h\}\setminus \{p\})\cup B_r(p)\,.
    \end{split}
\end{equation}
By \eqref{e:p_isolated} we have that $B_r(\supp(S)\cup\{p_1,\dots,p_h\}\setminus \{p\})$ and $B_{2r}(p)$ are disjoint. Define $\tilde S_{n_j} := S_{n_j} \trace B_{2r}(\{p\})$. 
By \cite[Lemma 28.5]{SimonLN} with $f(x)=\dist(x,\{p\})$ and \eqref{e:sn_dentro},
we have that for $j$ sufficiently large,
\begin{equation}
\partial \tilde S_{n_j} = b_{n_j} \trace B_{2r}(\{p\}) = -\frac{1}{k}\partial(T\trace B_{{n_j}^{-1}}(p)),
\end{equation}
which is supported in exactly two points.
Since necessarily $\tilde S_{n_j}\in\OTP(\partial \tilde S_{n_j})$, we deduce that 
\begin{equation}\label{e:snoutsidesupport}
\tilde S_{n_j} = -\frac{1}{k} T\trace B_{{n_j}^{-1}}(p),
\end{equation}
for $j$ sufficiently large.

Combining \eqref{e:snoutsidesupport}, \eqref{e:THE_ONE} and \eqref{e:alphamass_tienne}, we obtain that, for $j$ sufficiently large and for $k\geq k_0(\alpha)$
\begin{equation}\label{e:chofame}
\begin{split}
        \MM &\left(S_{n_j}+\frac1k \sum_{i=1}^h T\trace B_{{n_j}^{-1}}(p_i)\right)\\
        &=\MM(S_{n_j})+\sum_{i:S_{n_j}\trace Q'=W_j}(1-(1-k^{-1})^\alpha)\MM(T\trace B_{{n_j}^{-1}}(p_i))\\
        &\quad -\sum_{i:S_{n_j}\trace Q'=Z_j}k^{-\alpha}\MM(T\trace B_{{n_j}^{-1}}(p_i))-\sum_{i:p_i\notin\supp(S)}k^{-\alpha}\MM(T\trace B_{{n_j}^{-1}}(p_i))\\
        &\leq \MM(T_{n_j})+\sum_{i:S_{n_j}\trace Q'=W_j}(1-(1-k^{-1})^\alpha)\MM(T\trace B_{{n_j}^{-1}}(p_i))\\
        &\leq \MM(T_{n_j})+\sum_{i=1}^h (1-(1-k^{-1})^\alpha)\MM(T\trace B_{{n_j}^{-1}}(p_i))=\MM(T).
        \end{split}
\end{equation}
Observe that equality holds if and only if the negative terms above vanish which yields the validity of the claim in the case $S\neq T$. Moreover \eqref{e:chofame} trivially holds also in the case $S=T$, which concludes the validity of the claim in the general case, and of Lemma \ref{l:main}.
\end{proof}

\begin{proof}[Proof of Proposition \ref{p:unique_bienne}]
Since the conclusion of Lemma \ref{l:main} holds for every converging subsequence $S_{n_j}$, we deduce that $S_n=T_n$ and therefore $\OTP(b_n)=\{T_n\}$, for $n$ sufficiently large.
\end{proof}

\section{Proof of Theorem \ref{t:main}.}
By Lemma \ref{l:residual} it suffices to prove that the set $A_C\setminus NU_C$ is $\Flat_K$-dense in $A_C$. Fix $b\in A_C$ and $\varepsilon>0$. Let $\delta>0$ and $b''\in A_{C-\delta}$ be obtained by Lemma \ref{l:integral_bdry}. In particular let $b_I\in\I_0(K)$ be such that $b''=\eta b_I$ for some $\eta>0$.

Fix $T\in\OTP(b_I)$ and let $p_1,\dots,p_h$ be obtained applying Lemma \ref{l:magic_points} to the current $T$. Observe that $h$ depends on $T$. Let $k\in\N \setminus \{0\}$ be such that $k\geq k_0(\alpha)$ given by Proposition \ref{p:unique_bienne} and moreover $hk^{-1}C\leq\eta^{-1}\delta$.
For $n=1,2,\dots$, let $b_n$ be obtained as in \eqref{e:bienne}, where $b$ is replaced with $b_I$. By \eqref{e:massa_bienne}, for every $n$ we have 
$$\Mass(\eta b_n) = \eta\Mass(b_n)\leq\eta(\Mass(b_I)+hk^{-1}C)\leq\eta(\eta^{-1}(C-\delta)+\eta^{-1}\delta)=C.$$
Moreover, letting $S_n\in\OTP(b_n)$, by \eqref{e:alphamass_tienne} we have $\MM(\eta S_n)\leq\MM(\eta T)\leq C-\delta$, which allows to conclude that $\eta b_n\in A_C$ for every $n\in\N \setminus \{0\}$. By Proposition \ref{p:unique_bienne} we deduce that $\eta b_n\in A_C\setminus NU_C$ for $n$ sufficiently large, and by \eqref{e:convergence_bienne}, we have $$\Flat_K(\eta b_n-b)\leq\Flat_K(\eta b_n-\eta b_I)+\Flat_K(b''- b) <2\varepsilon,$$ 
for $n$ sufficiently large. By the arbitrariness of $\varepsilon$ we conclude the proof of the density of $A_C\setminus NU_C$ and hence the proof of the Theorem.

\appendix
\section{Improved stability}
The aim of this section is to improve the main result of \cite{CDRMcpam} by proving the following result.
\begin{theorem}\label{t:stability_new}
Let $b_n\in A_C$, see \eqref{def_AC}, and let $S_n \in \OTP(b_n)$. For every subsequential limit $T$ of $S_n$ we have $T\in\OTP(\partial T)$.
\end{theorem}
\begin{proof}
The subsequential convergence $\Flat(S_n-T)\to 0$ implies $\Flat(b_n-\partial T)\to 0$ and writing $b_n=\mu^+_n-\mu^-_n$ (being $\mu^+_n$ and $\mu^-_n$ respectively the positive and the negative part of the signed measure $b_n$) and $\mu^\pm:=\lim_{n\to\infty}\mu^\pm_n$, we have $\partial T=\mu^+-\mu^-$, where $\mu^+$ and $\mu^-$ are not necessarily mutually singular.

Hence, with respect to \cite[Theorem 1.1]{CDRMcpam} we simply need to remove the assumption that $\mu^-$ and $\mu^+$ are mutually singular. In fact we observe that such assumption does not have a fundamental role in the proof already given in \cite{CDRMcpam} and, more precisely, we analyze all the points where such assumption is relevant.
\begin{itemize}
    \item In \cite[equation (4.9)]{CDRMcpam} the assumption is used, but we observe that if we do not assume that $\mu^-$ and $\mu^+$ are mutually singular, \cite[equation (4.9)]{CDRMcpam} would be replaced by 
    $$\partial T^{ij}=\int_{\Lip(Q^i,Q^j)}\delta_{\gamma(\infty)}-\delta_{\gamma(0)}dP(\gamma),$$
    which suffices to obtain \cite[equation (4.23)]{CDRMcpam}, which is the only point where \cite[equation (4.9)]{CDRMcpam} is (implicitly) used.   
    \item In \cite[pag. 852, line 7]{CDRMcpam}, the fact that $\mu^-$ and $\mu^+$ are mutually singular is actually not necessary.
    \item The fact that $\mu^-$ and $\mu^+$ are mutually singular is necessary to obtain \cite[equations (4.16), (4.17)]{CDRMcpam} and more precisely without such assumption the validity of those equations might fail in the cubes $\{Q^h:h=1,\dots,N\}$ but it remains true (with the same argument) in the remaining cubes $Q^i\in\Lambda(Q,k)$. However, we observe that \cite[equations (4.16), (4.17)]{CDRMcpam} are only used to obtain \cite[equation (4.18)]{CDRMcpam}, which remains valid, precisely because it is stated only for the cubes $Q^i\in\Lambda(Q,k)\setminus \{Q^h:h=1,\dots,N\}$.
\end{itemize}
In conclusion, with the minor modifications listed above, the proof of \cite[Theorem 1.1]{CDRMcpam} remains valid even without the assumption that $\mu^-$ and $\mu^+$ are mutually singular, thus concluding our proof. 
\end{proof}

\newpage

\subsection*{Acknowledgments}
A.M. acknowledges partial support from PRIN 2017TEXA3H\_002 "Gradient flows, Optimal Transport and Metric Measure Structures". 
\subsection*{Data availability statement}
Data sharing not applicable to this article as no dataset were generated or analysed during the current study.
%
%

\nocite{}

%
%
%
%
%
\vskip 1 cm

{\parindent = 0 pt\begin{footnotesize}

Gianmarco Caldini
\\
Dipartimento di Matematica, Universit\`a degli Studi di Trento.\\
e-mail: {\tt gianmarco.caldini@unitn.it}
\\
\\
Andrea Marchese
\\
Dipartimento di Matematica, Universit\`a degli Studi di Trento.\\
e-mail: {\tt andrea.marchese@unitn.it}
\\
\\
Simone Steinbr\"uchel
\\
Mathematisches Institut, Universit\"at Leipzig.\\
e-mail: {\tt simone.steinbruechel@math.uni-leipzig.de}
\end{footnotesize}
}

\end{document}